\documentclass[12pt]{article}

\usepackage{mathtools}
\usepackage{listings}
\usepackage{amsthm,amsmath,amssymb}
\usepackage{ wasysym } 
\usepackage{graphicx}
\usepackage{multirow}
\usepackage{adjustbox}
\usepackage{pdflscape}
\usepackage{float}
\usepackage{mathtools}
\usepackage{amsfonts}
\usepackage{color}
\usepackage{calc}

\theoremstyle{plain}
\newtheorem{theorem}{Theorem}[section]
\newtheorem{lemma}[theorem]{Lemma}
\newtheorem{corollary}[theorem]{Corollary}

\graphicspath{{figures/}}

\DeclarePairedDelimiter\I{[\![}{]\!]}
\DeclarePairedDelimiter\abs{\lvert}{\rvert}
\DeclarePairedDelimiter\step{\langle}{\rangle}

\newcommand{\MT}{{\bigtriangleup}_M} 
\newcommand{\TT}{{\bigtriangleup}_T} 

\newcommand{\SVMS}{\mathfrak{M}} 
\newcommand{\SVMW}{\mathfrak{A}} 
\newcommand{\MWSet}{\mathbf{A^{Q}}}  
\newcommand{\MWRSet}{\mathbf{A^{Q}_{R}}} 

\newcommand{\MFlatCount}{M_{\invdiameter}^{Q}}

\newcommand{\GMFlatCount}{M_{\invdiameter}^{H}}
\newcommand{\GMWSet}{\mathbf{A^{H}}} 
\newcommand{\GMWRSet}{\mathbf{A^{H}_{R}}} 
\newcommand{\MSCount}{M^{Q}} 
\newcommand{\MWCount}{A^{Q}}  
\newcommand{\MWRCount}{A^{Q}_{R}} 
\newcommand{\GMSCount}{M^{H}} 
\newcommand{\GMWCount}{A^{H}} 
\newcommand{\GMWRCount}{A^{H}_{R}} 

\newcommand{\SVDS}{\mathfrak{D}} 
\newcommand{\SVDW}{\mathfrak{B}} 
\newcommand{\DWSet}{\mathbf{B^{Q}}}  
\newcommand{\DWRSet}{\mathbf{B^{Q}_{R}}} 
\newcommand{\GDWSet}{\mathbf{B^{H}}} 
\newcommand{\GDWRSet}{\mathbf{B^{H}_{R}}} 
\newcommand{\DWCount}{B^{Q}}  
\newcommand{\DWRCount}{B^{Q}_{R}} 
\newcommand{\GDWCount}{B^{H}} 
\newcommand{\GDWRCount}{B^{H}_{R}} 

\newcommand{\SVCS}{\mathfrak{S}} 
\newcommand{\SVCW}{\mathfrak{C}} 
\newcommand{\CWSet}{\mathbf{C^{Q}}}  
\newcommand{\CWRSet}{\mathbf{C^{Q}_{R}}} 
\newcommand{\GCWSet}{\mathbf{C^{H}}} 
\newcommand{\GCWRSet}{\mathbf{C^{H}_{R}}} 
\newcommand{\CWCount}{C^{Q}}  
\newcommand{\CWRCount}{C^{Q}_{R}} 
\newcommand{\GCWCount}{C^{H}} 
\newcommand{\GCWRCount}{C^{H}_{R}} 
\newcommand{\ssd}{Smooth } 

\newcommand{\mflat}{m_{\invdiameter}^{Q}}
\newcommand{\gmflat}{m_{\invdiameter}^{H}}
\newcommand{\mw}{a^{Q}}
\newcommand{\mwr}{a^{Q}_{R}}
\newcommand{\gmw}{a^{H}}
\newcommand{\gmwr}{a^{H}_{R}}

\newcommand{\dw}{b^{Q}}
\newcommand{\dwr}{b^{Q}_{R}}
\newcommand{\gdw}{b^{H}}
\newcommand{\gdwr}{b^{H}_{R}}

\newcommand{\gdflat}{d_{\invdiameter}^{H}}
\newcommand{\cw}{c^{Q}}
\newcommand{\cwr}{c^{Q}_{R}}
\newcommand{\gcw}{c^{H}}
\newcommand{\gcwr}{c^{H}_{R}}

\usepackage{xspace}
\newcommand{\vc}{vertically constrained\xspace}
\newcommand{\VC}{Vertically constrained\xspace}

\newcommand\g[1]{
    {\footnotesize \emph{#1}}
}
\newcommand\s[1]{
    \footnotesize{#1}
}
\newcommand\q[1]{
    \scriptsize{#1}
}

\begin{document}

\title{Vertically constrained Motzkin-like paths inspired by bobbin lace}%

\author{Veronika Irvine\thanks{Supported by NSERC PDF}\\
\small Cheriton School of Computer Science\\[-0.8ex]
\small University of Waterloo\\[-0.8ex]
\small Waterloo ON, Canada\\
\small\tt virvine@uwaterloo.ca\\
\and
Stephen Melczer\thanks{Supported by NSERC PDF and NSF Grant DMS-1612674}\\
\small Department of Mathematics\\[-0.8ex]
\small University of Pennsylvania \\[-0.8ex]
\small Philadelphia, PA\\
\and
Frank Ruskey\thanks{Supported by NSERC Discovery Grants RGPIN-3379-2009 and RGPIN-2014-04883}\\
\small Department of Computer Science\\[-0.8ex]
\small University of Victoria\\[-0.8ex]
\small Victoria BC, Canada\\
}

\date{April 13, 2019}

\maketitle

\begin{abstract}
Inspired by a new mathematical model for bobbin lace, this paper considers finite lattice paths formed from the set of step vectors $\SVMW=$$\{\rightarrow,$ $\nearrow,$ $\searrow,$ $\uparrow,$ $\downarrow\}$ with the restriction that vertical steps $(\uparrow, \downarrow)$ cannot be consecutive. The set $\SVMW$ is the union of the well known Motzkin step vectors $\SVMS=$$\{\rightarrow,$ $\nearrow,$ $\searrow\}$ with the vertical steps $\{\uparrow, \downarrow\}$. An explicit bijection $\phi$ between the exhaustive set of \vc paths formed from $\SVMW$ and a bisection of the paths generated by $\SVMS$ is presented. In a similar manner, paths with the step vectors $\SVDW=$$\{\nearrow,$ $\searrow,$ $\uparrow,$ $\downarrow\}$, the union of Dyck step vectors and constrained vertical steps, are examined.  We show, using the same $\phi$ mapping, that there is a bijection between \vc $\SVDW$ paths and the subset of Motzkin paths avoiding horizontal steps at even indices.  Generating functions are derived to enumerate these \vc, partially directed paths when restricted to the half and quarter-plane.  Finally, we extend Schr\"{o}der and Delannoy step sets in a similar manner and find a bijection between these paths and a subset of Schr\"{o}der paths that are smooth (do not change direction) at a regular horizontal interval.
\end{abstract}

\section{Introduction}

\begin{figure}
\centering
    \def\svgwidth{0.6\textwidth}%
    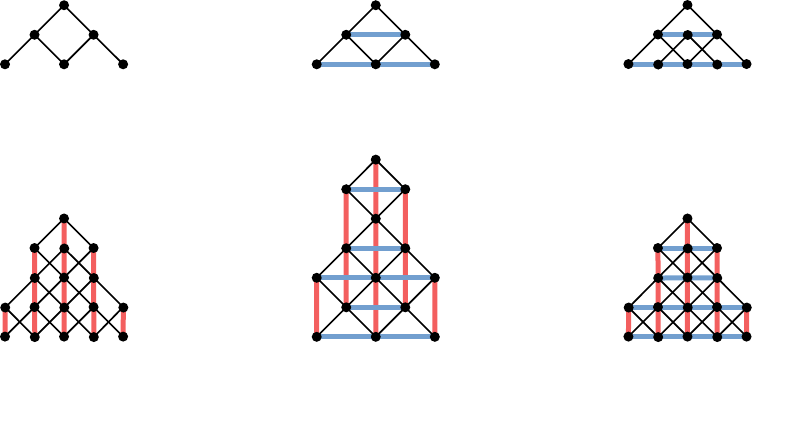%

\caption{Underlying graphs induced by the complete set of paths in the quarter-plane for some well known directed lattice path classes and their \vc extensions.}
\label{fig:summary}
\end{figure}

Lattice path enumeration is a classic part of combinatorial enumeration, having been explored for over a century.
Research in this area has resulted in well known classes of lattice paths such as those named after Dyck, Motzkin, Schr\"{o}der and Delannoy~\cite{donaghey,Stanley2012, barcuccimotzkin}.  Many variations on lattice paths have been examined, and we refer the reader to the survey by Humphreys~\cite{humphreys} for additional historical information.

In the process of creating a mathematical model for a fiber art form known as bobbin lace~\cite{irvine}, we encountered lattice paths very similar to the Motzkin paths but with the addition of vertical steps.  An example of one of these partially directed (weakly monotonic in the $x$ direction and self avoiding) paths is shown in Figure \ref{fig:laceconnection}.  In recent work, Dziemia\'nczuk~\cite{dziemianczuk} examined lattice paths with a vertical step in the down direction $(\downarrow)$.  When only down steps are allowed, the number of paths between the origin and a termination point remains finite. With the addition of both up and down step vectors, an infinite number of paths are possible.  Traditionally in such cases the number of paths is restricted to a finite value by considering only paths with a fixed number of steps (see for example the ``slow walk'' example of Niederhausen and Heinrich~\cite{niederhausen}) or by constraining the walk to a certain region of space such as a wedge or a slit~\cite{van2008partially}.   In contrast, for the self-avoiding lattice paths encountered in bobbin lace, the constraint that limits the number of walks ending at a specified point to a finite number is that vertical steps $(\uparrow, \downarrow)$ cannot be consecutive.

The significance of this constraint on vertical steps can be understood by looking at the role of the paths in bobbin lace tessellations.  A drawing embedded on the torus of a 2-regular directed graph $D(V,E)$ can be used to model the placement of threads in a lace tessellation~\cite{irvine}.  In order to describe a workable lace pattern, $D(V,E)$ must have the property that it can be partitioned into a set of osculating paths: paths that meet at a vertex but do not cross transversely. Furthermore, at every vertex of $D(V,E)$, two paths must meet.  It is not possible to combine two weakly monotonic  lattice path in this manner if one or both of the paths contains consecutive vertical steps.

\begin{figure}[H]
\centering
    \def\svgwidth{0.8\textwidth}%
    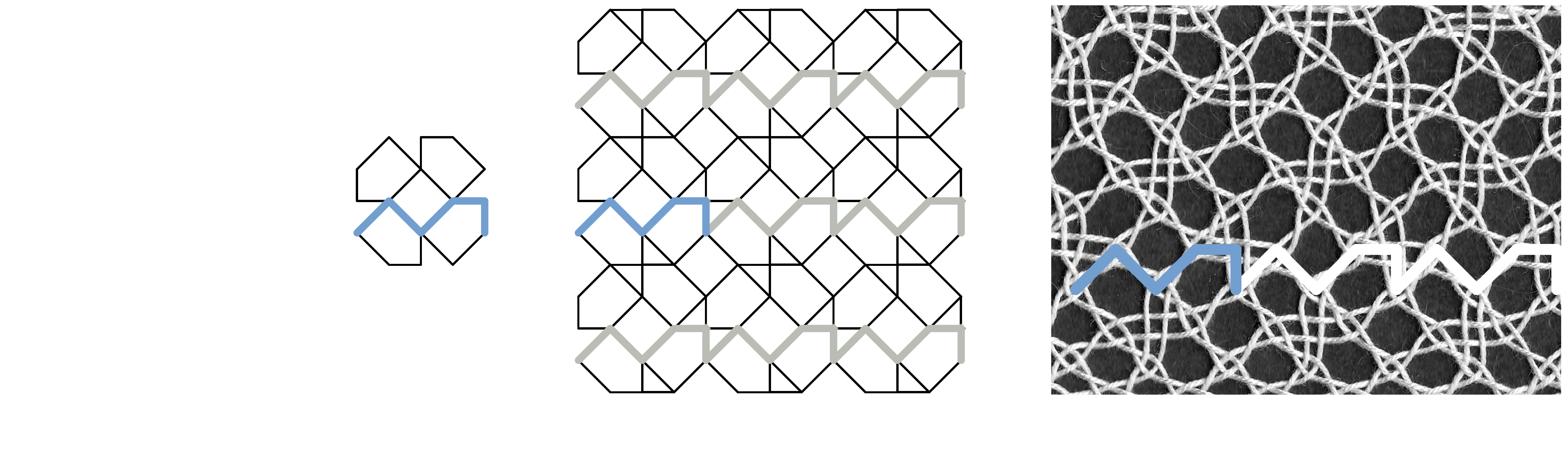%

\caption{Bobbin lace pattern: a) Diagram of threads. b) Unit of pattern. c) Periodic tiling of pattern. d)  Execution of lace pattern in cotton thread.}
\label{fig:laceconnection}
\end{figure}

The purpose of this study is to examine \vc lattice paths on the discrete Cartesian plane $\mathbb{Z} \times \mathbb{Z}$ constructed from the finite set of step vectors $\SVMW=\{\step{1,0}, \step{1,1},$ $\step{1,-1}, \step{0,1}, \step{0,-1} \}$.  Here $\SVMW$ is the union of the set of step vectors used in Motzkin paths with the set of up and down step vectors.  We use the term \emph{\vc} to indicate paths that do not have consecutive vertical steps.  That is, the following sequences are not permitted: $\uparrow\uparrow$, $\downarrow\downarrow$, as well as the self intersecting sequences $\uparrow\downarrow$ and $\downarrow\uparrow$.

In Section \ref{sec:recurrence} we describe the recurrence relation for \vc $\SVMW$ paths. In Section \ref{sec:bijection} we  identify a bijection between \vc $\SVMW$ paths and Motzkin paths. In Section \ref{sec:genfunction} we derive a generating function for $\SVMW$ paths in the half-plane making use of the recurrence relation, and in the quarter-plane using the bijection.  The same approach is applied in Section \ref{sec:extensions} to extend both the Dyck and Schr\"{o}der classes of paths. We conclude with suggestions for further exploration of this family of partially directed, vertically constrained lattice paths.

\section{Recurrence Relations}
\label{sec:recurrence}
The lattice paths discussed in this paper start at the origin and travel in a weakly monotonic manner to the right. A point that is a horizontal distance $n$ and vertical distance $m$ from the origin will be represented as $(n,m)$.

As shown in Table \ref{table:cases}, the lattice paths can be divided into four classes using two boolean properties:

1) Can the path extend below the $x$-axis? That is, is the path restricted to the upper right quarter-plane (Q) or can it travel anywhere in the right half-plane (H)?

2) Can the leading step be a vertical step vector?  In other words, is the leading step, $e_1$, chosen from the set $\SVMW$ (a leading vertical step is allowed) or from the set $\SVMS$ (leading step is not vertical)?

The importance of this second property will become clear when we describe a bijection with Motzkin paths in Section \ref{sec:bijection}.

\begin{table}[H]
\begin{center}
\begin{tabular}{|l|l|l|l|}
\hline
Name & Restricted to Q & Leading Step & Description \\ \hline \hline
$\GMWSet$   & False & $e_1 \in \SVMW$ & Leading vertical steps allowed \\ \hline
$\GMWRSet$  & False & $e_1 \in \SVMS$ & No leading vertical steps \\ \hline
$\MWSet$    & True  & $e_1 \in \SVMW$ & Restricted to quarter-plane, \\
            &       &                 & leading vertical steps allowed \\ \hline
$\MWRSet$   & True  & $e_1 \in \SVMS$ & Restricted to quarter-plane, \\
            &       &                 & no leading vertical steps \\ \hline
\end{tabular}
\end{center}
\caption{Four classes of \vc lattice paths}
\label{table:cases}
\end{table}

We define $\GMWSet$ to be the set of partially directed, \vc lattice paths in the half-plane created from step vectors $\SVMW$.  Similarly, we denote the number of paths of type $\GMWSet$ starting at point $(0,0)$ and terminating at point $(n, m)$ by $\GMWCount(n,m)$.

For a proposition $P$, let $\I{P}$ be 1 if $P$ is true and 0 otherwise.

\begin{lemma}
$\GMWCount(n,m)$ satisfies the recurrence relation
\begin{align}
\GMWCount(0,m) &= \I{ m \in \{-1,0,1\} }\label{eq:recurrenceAInitial} \\
\GMWCount(n,m) &= \GMWCount(n{-}1,m{+}2) + 2\GMWCount(n{-}1,m{+}1) + 3\GMWCount(n{-}1,m) \nonumber \\
       &\quad{+} \, 2\GMWCount(n{-}1,m{-}1) + \GMWCount(n{-}1,m{-}2). \label{eq:recurrenceA}
\end{align}
\label{theo:recurrenceA}
\end{lemma}
\begin{proof}
Equation \eqref{eq:recurrenceA} is obtained by examining all of the ways in which a $\GMWSet$ path can terminate at a lattice point as illustrated in Figure \ref{fig:validarcs}.
\end{proof}
Such a recursive counting formula is common in lattice path enumeration and adjacent topics.  See, for instance, the work of Donaghey and Shapiro~\cite{donaghey} on the Motzkin triangle.

\begin{figure}[H]
\centering
    \def\svgwidth{0.4\textwidth}%
    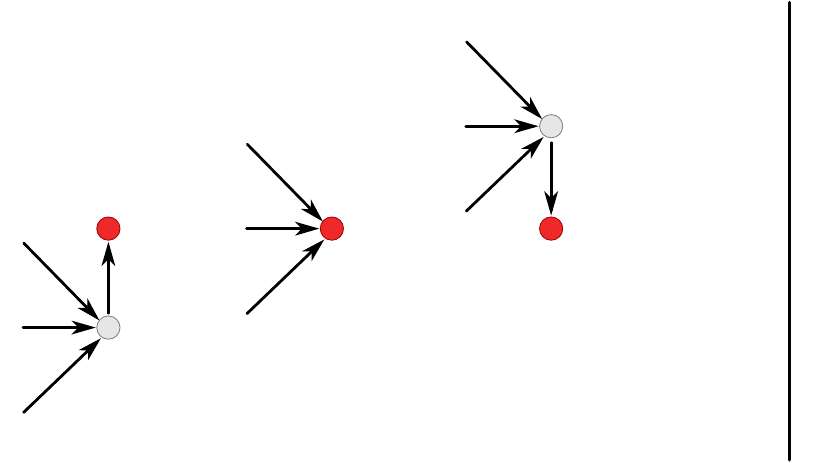%

\caption{All possible ways in which a vertically constrained path with steps from the set $\SVMW$ can terminate at a lattice point(indicated by red dot).}
\label{fig:validarcs}
\end{figure}

In a similar manner, we define $\GMWRSet$ as the set of partially directed, \vc lattice paths in the half-plane created from step vector $\SVMW$ in which the leading step is restricted to $\SVMS$. The number of paths that travel from the origin to point $(n,m)$, is represented by $\GMWRCount(n,m)$.  The initial condition is $\GMWRCount(0,m) = \I{m = 0}$.  Equation \eqref{eq:recurrenceA} can be rewritten for $\GMWRCount(n,m)$ by substituting $\GMWRCount(a,b)$ for $\GMWCount(a,b)$.

Paths that are restricted to the quarter-plane require special handling for positions close to the $x$-axis.  We define $\MWSet$ as the set of partially directed, \vc lattice paths restricted to the quarter-plane, created from step vectors $\SVMW$.  The count of $\MWSet$ paths terminating at $(n,m)$ is given by $\MWCount(n,m)$.
\begin{lemma}
The recurrence relation for $\MWCount(n,m)$ is
\begin{align}
\MWCount(0,m) &= \I{m \in \{0,1\}} \\
\MWCount(n,0) &= \MWCount(n{-}1,2)     + 2\MWCount(n{-}1,1)     + 2\MWCount(n{-}1,0) \label{eq:r1} \\
\MWCount(n,1) &= \MWCount(n{-}1,3)     + 2\MWCount(n{-}1,2)     + 3\MWCount(n{-}1,1) + 2\MWCount(n{-}1,0) \label{eq:r2}\\
\MWCount(n,m) &= \MWCount(n{-}1,m{+}2) + 2\MWCount(n{-}1,m{+}1) + 3\MWCount(n{-}1,m) \nonumber \\
       &\quad {+} \, 2\MWCount(n{-}1,m{-}1) + \MWCount(n{-}1,m{-}2). \label{eq:r3}
\end{align}
\label{theo:recurrenceAQ}
\end{lemma}
This follows again from considering all ways in which a walk can end at a specific point.

The recurrence relation for the set of $\MWRSet$ paths, in which the leading step is restricted to $\SVMS$, has initial condition $\MWRCount(0,m) = \I{m = 0}$. Equations \eqref{eq:r1},~\eqref{eq:r2} and \eqref{eq:r3} can be rewritten for $\MWRSet$ paths by substituting $\MWRCount(a,b)$ for $\MWCount(a,b)$.

\section{An Explicit Bijection}
\label{sec:bijection}

\begin{table}[H]
\begin{center}
\begin{adjustbox}{max width=\textwidth}
\begin{tabular}{|l|rl|l|}
\hline
\multicolumn{1}{|c|}{Motzkin family}   & \multicolumn{2}{|c|}{Relationship}          & \multicolumn{1}{|c|}{OEIS~\cite{oeis}} \\ \hline \hline
\multirow{2}{4cm}{Motzkin Paths}       & $\MWRCount(n,0)$  &$= \MSCount(2n)$         & Bisection of A001006 (also A026945) \\ \cline{2-4}
                                       & $\MWCount(n,0)$   &$= \MSCount(2n{+}1)$     & Bisection of A001006 (also A099250) \\ \hline
\multirow{2}{4cm}{Grand Motzkin Paths} & $\GMWRCount(n,0)$ &$= \GMSCount(2n)$        & Bisection of A002426 (also A082758)\\ \cline{2-4}
                                       & $\GMWCount(n,0)$  &$= \GMSCount(2n{+}1)$    & Bisection of A002426 \\ \hline
\multirow{2}{4cm}{Motzkin Triangle}    & $\MWRCount(n,m)$  &$= \MT(2n, m)$           & Row Bisection of  A026300 \\ \cline{2-4}
                                       & $\MWCount(n,m)$   &$= \MT(2n{+}1,m)$        & Row Bisection of  A026300 \\ \hline
\multirow{2}{4cm}{Triangle of Trinomial Coefficients}      &$\GMWRCount(n,m)$ &$= \TT(2n, m)$    & Row Bisection of A027907 \\ \cline{2-4}
                                                           &$\GMWCount(n,m)$ &$= \TT(2n{+}1, m)$ & Row Bisection of A027907 \\ \hline
\end{tabular}
\end{adjustbox}
\end{center}
\caption{Correlation between \vc paths and their Motzkin counterparts.  Bisection means mapping to alternating (even or odd) indices or rows of a sequence.}
\label{table:summary2}
\end{table}

Table \ref{table:summary2} lists some well-known lattice path models.  Both Motzkin and Grand Motzkin paths use the $\SVMS$ step vectors and start and end on the $x$-axis; the distinction is that Motzkin paths are restricted to the quarter-plane while Grand Motzkin paths are restricted to the half-plane. The Motzkin triangle enumerates $\SVMS$ paths in the quarter-plane that terminate at a more general point $(x,y)$ and the Triangle of Trinomial Coefficients does the same for $\SVMS$ paths in the half-plane.

As indicated in Table \ref{table:summary2}, there is a relationship between the \vc $\SVMW$ lattice paths and a bisection of the Motzkin paths. We prove this relationship by providing an explicit bijection.

We define the mapping $\phi$, illustrated in Figure \ref{fig:map1},  in which the following substitutions are performed on alternating steps of the Motzkin path:
\begin{eqnarray*}
   \phi :& \step{1,1}  & \to  \step{0,1}\\
   \phi :& \step{1,-1} & \to  \step{0,-1}\\
   \phi :& \step{1,0}  & \to \text{remove step}
\end{eqnarray*}

\begin{figure}[H]
\centering
    \def\svgwidth{0.9\textwidth}%
    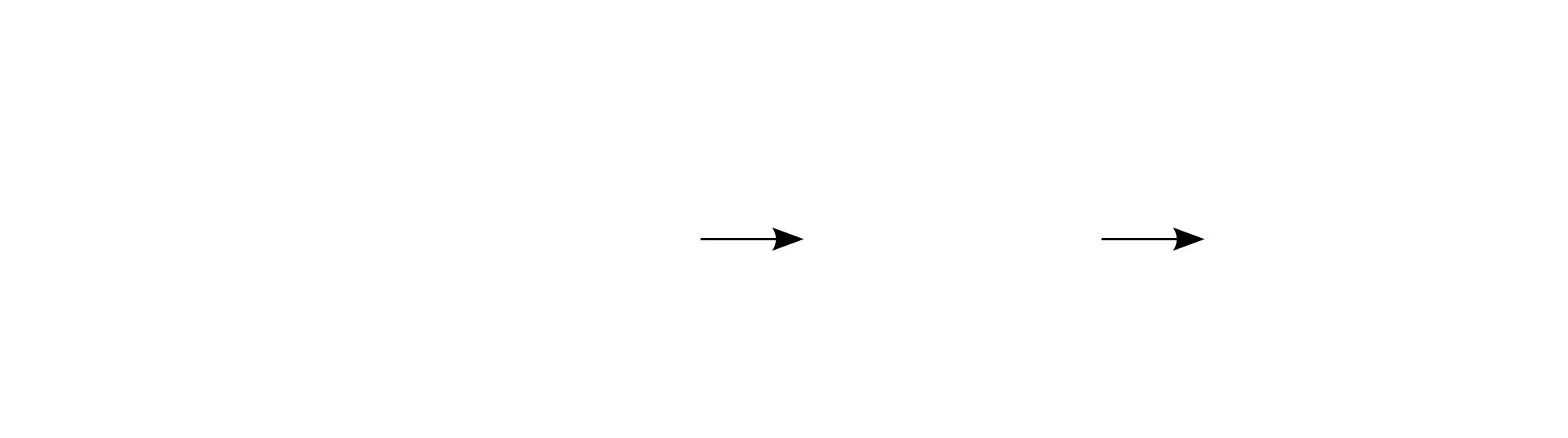%

\caption{Conversion of a Motzkin path ($n=7$) to an $\MWSet$ path ($n=3$).  Grey steps are modified using the substitution rules of $\phi$. Black steps are unchanged.}
\label{fig:map1}
\end{figure}

\begin{lemma}
The mapping $\phi$ applied to alternating steps starting at the \emph{first} step of a Motzkin (Grand Motzkin) path terminating at $(2n+1,m)$  produces a \vc path of type $\MWSet$ $(\GMWSet)$ terminating at $(n,m)$.
\label{theo:phi}
\end{lemma}
\begin{proof}
Step vectors in the Motzkin paths belong to $\SVMS$ which is a subset of $\SVMW$.  All substituted step vectors in the mapping $\phi$ belong to $\SVMW$, therefore the resulting path contains only step vectors from $\SVMW$.  Only alternating steps of the Motzkin path are modified, therefore consecutive vertical step vectors cannot be introduced by $\phi$. The resulting path is thus a \vc path of type $\MWSet$ $(\GMWSet)$. At each substitution, the horizontal length of the path is decreased by $1$. The mapping starts with a path of length $2n+1$ and horizontally collapses $n+1$ steps resulting in a path of length $n$.  Only horizontal displacement is affected by the substitutions of $\phi$ leaving the height of vertices along the path unchanged.  Therefore, an operand restricted to one quadrant produces a result restricted to one quadrant and an operand that terminates at height $m$ produces a result that terminates at height $m$.
\end{proof}
\begin{lemma}
The mapping $\phi$ applied to alternating steps starting at the \emph{second} step of a Motzkin (Grand Motzkin) path terminating at $(2n,m)$ produces a path of type $\MWRSet$ $(\GMWRSet)$ (restricted leading step) terminating at $(n,m)$.
\label{theo:phi2}
\end{lemma}
\begin{proof}
The same arguments used in Lemma \ref{theo:phi} apply.  Because the mapping starts with the second step of the Motzkin path, the first step of the resulting path belongs to $\SVMS$ and $n$ steps are collapsed horizontally.
\end{proof}

The inverse relationship, taking a \vc $\SVMW$ path to a Motzkin path, can be represented by the mapping $\psi$, illustrated in Figure \ref{fig:unmap1}, in which the following substitutions are performed on the vertical steps at each integral horizontal distance from the origin.

\begin{eqnarray*}
\psi : & \step{0,1}  & \to \step{1,1} \\
\psi : & \step{0,-1} & \to \step{1,-1} \\
\psi : & \text{no vertical step}  & \to \text{ insert } \step{1,0}
\end{eqnarray*}

\begin{figure}[H]
\centering
    \def\svgwidth{0.8\textwidth}%
    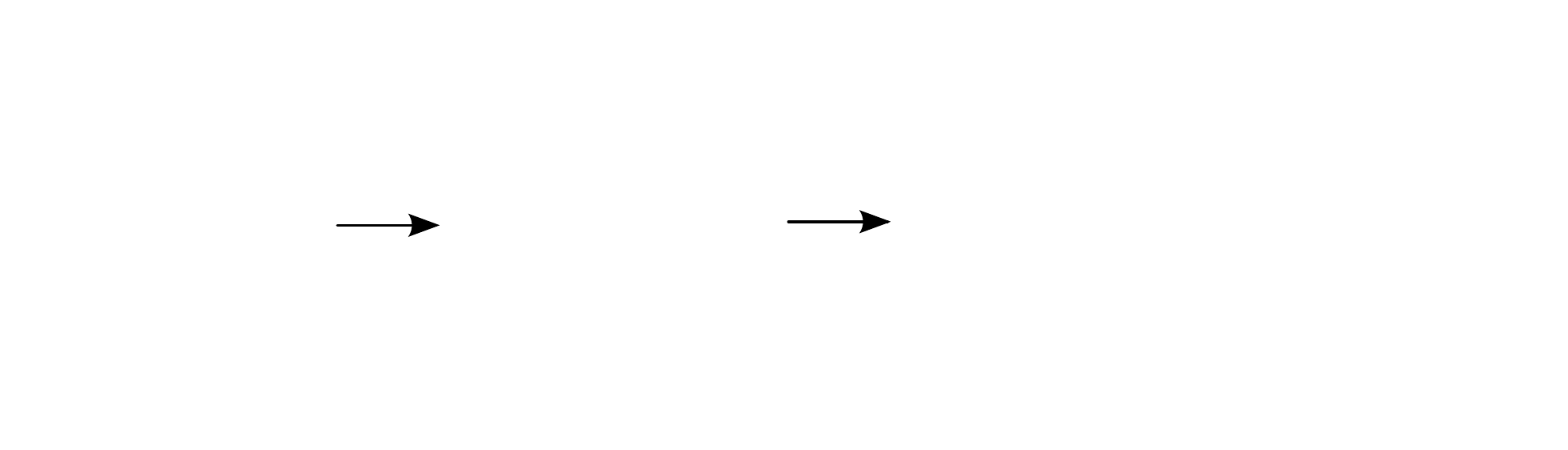%

\caption{Conversion of $\GMWSet$ path ($n=3$) to Grand Motzkin path ($n=7$)}
\label{fig:unmap1}
\end{figure}

\begin{lemma}
The mapping $\psi$ applied at a horizontal distance $i$ from the origin for all $0 \le i \le n$ to a \vc $\MWSet$ $(\GMWSet)$ path terminating at $(n,m)$  produces a Motzkin (Grand Motzkin) path terminating at $(2n+1,m)$.
\label{theo:psiA}
\end{lemma}
\begin{proof}
The mapping $\psi$ replaces all vertical steps with steps belonging to $\SVMS$. The resulting path is therefore a Motzkin or Grand Motzkin path. The substitutions do not affect the height of vertices along the path, therefore an input restricted to a quadrant produces an output restricted to a quadrant and an input that terminates at height $m$ produces an output that terminates at height $m$.  At each substitution, the horizontal length of the path is increased by $1$.  Since $n+1$ substitutions are performed on an operand of length $n$, the resulting path has length $2n+1$.
\end{proof}

\begin{lemma}
The mapping $\psi$ applied at a horizontal distance $i$ from the origin for all $0 < i \le n$ to a \vc $\MWRSet$ $(\GMWRSet)$ path terminating at $(n,m)$  produces a Motzkin (Grand Motzkin) path terminating at $(2n,m)$.
\end{lemma}
\begin{proof}
The proof for this lemma is similar to the proof for Lemma \ref{theo:psiA}. Because the substitutions begin at $i=1$, the number of substitutions is $n$ and the length of the resulting path is $2n$.
\end{proof}

\begin{theorem}
The map $\phi$ is a bijection from Motzkin (Grand Motzkin) paths of length $2n+1$ to \vc paths of type $\MWSet$ $(\GMWSet)$ of length $n$, with inverse $\psi$.
\label{theo:bijectionA}
\end{theorem}
\begin{proof}
First we prove that the mapping $\phi$ is an injection.  The input for $\phi$ is a Motzkin path.  Through $\phi$, each step vector in the input is either copied directly to the output or uniquely mapped to a step vector that is not an element in $\SVMS$.  The same destination path could not be generated by distinct inputs, therefore, $\phi$ is an injection from a Motzkin (Grand Motzkin) path to a \vc $\MWSet$ $(\GMWSet)$ path.

Similarly, $\psi$ is an injection from vertically constrained paths of type \vc $\MWSet$ $(\GMWSet)$ to the collection of Motzkin (Grand Motzkin) paths of length $2n+1$. This implies $\phi$ is a bijection, and it is clear that its inverse is $\psi$.
\end{proof}

\begin{theorem}
The map $\phi$ is a bijection from Motzkin (Grand Motzkin) paths of length $2n$ to \vc paths of type $\MWRSet$ $(\GMWRSet)$ and length $n$, with inverse $\psi$.
\end{theorem}
\begin{proof}
This can be proven using the same approach as Theorem \ref{theo:bijectionA}.
\end{proof}

Figure \ref{fig:exampleMotzkin} shows that \vc paths with a restricted leading step are a subset of the unrestricted \vc paths of the same horizontal length.  By the bijective mapping $\phi$, this relationship also exits between even and odd length Motzkin paths.  If we prepend a horizontal step, Motzkin paths of length $2n$ form a unique subset of Motzkin paths of length $2n+1$.

\begin{figure}[H]
\centering
    \def\svgwidth{\textwidth}%
    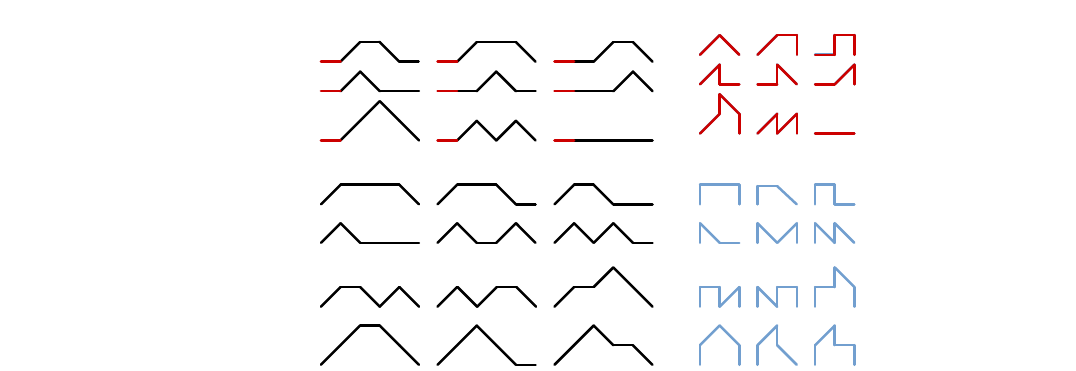%

\caption{Motzkin paths ($n=4$ and $n=5$) and corresponding $\MWRSet$ and $\MWSet$ paths ($n=2$).}
\label{fig:exampleMotzkin}
\end{figure}

\section{Generating Functions}
\label{sec:genfunction}

\begin{theorem}
For $\GMWSet$ and $\GMWRSet$, the generating functions are
\begin{align*}
\gmwr(x,y) &:= \sum_{n \ge 0} \sum_{m\in\mathbb{Z}} \GMWRCount(n,m) x^n y^m = \frac{y^2}{y^2{-}x \left(1 {+} y {+} y^2 \right)^2},\\
\gmw(x,y)  &:= \sum_{n \ge 0} \sum_{m\in\mathbb{Z}} \GMWCount(n,m) x^n y^m  = \frac{y\left(1{+}y{+}y^2\right)}{y^2{-}x \left(1 {+} y{+}y^2\right)^2},\\
\gmwr(x,0) &:= \sum_{n \ge 0} \GMWRCount(n,0) x^n = \frac{1}{P(x)}\sqrt{\frac{1{-}3x{+}P(x)}{2}},\\
\intertext{and}
\gmw(x,0)  &:= \sum_{n \ge 0} \GMWCount(n,0) x^n = \frac{1}{P(x)}\sqrt{\frac{1{-}3x{-}P(x)}{2x}}
\end{align*}
where
\[ P(x) := \sqrt{1{-}10x{+}9x^2}.\]
\end{theorem}
\begin{proof}
The generating functions for $\gmwr(x,y)$ and $\gmw(x,y)$ were derived by expanding terms in general summation using the recurrence relations of Lemma \ref{theo:recurrenceA}. See, for example, Wilf~\cite[p.~5]{wilf}.  The expressions for $\gmwr(x,0)$ and $\gmw(x,0)$ were obtained by bisecting the well known generating function for Grand Motzkin paths~\cite{barcuccimotzkin} and simplifying.
\end{proof}

\begin{corollary}
For $\MWSet$ and $\MWRSet$, the generating functions are as follows:
\begin{align*}
\mwr(x,y) &:= \sum_{n \ge 0} \sum_{m \ge 0} \MWRCount(n,m) x^n y^m
            = \frac{R(x)+Q(x)+2}{\left(1{+}Q(x){-}\sqrt{x}(1{+}2y)\right)\left(1{+}R(x)+\sqrt{x}(1{+}2y)\right)},\\
\mw(x,y)  &:= \sum_{n \ge 0} \sum_{m \ge 0} \MWCount(n,m) x^n y^m
            = \frac{R(x)-Q(x)+2\sqrt{x}(1{+}2y)}{\sqrt{x}\left(1{+}Q(x){-}\sqrt{x}(1{+}2y)\right)\left(1{+}R(x)+\sqrt{x}(1{+}2y)\right)},\\
\mwr(x,0) &:= \sum_{n \ge 0} \MWRCount(n,0) x^n = \frac{1}{2x}\left(1 - \sqrt{ \frac{1{-}3x{+}P(x)}{2}}\right)\\
\intertext{and}
\mw(x,0)  &:= \sum_{n \ge 0} \MWCount(n,0) x^n = \frac{1}{2x}\left(\sqrt{\frac{1{-}3x{-}P(x)}{2x}} - 1\right)\\
\intertext{where}
Q(x) &:= \sqrt{1{-}2\sqrt{x}-3x}\\
R(x) &:= \sqrt{1{+}2\sqrt{x}-3x}\\
P(x) &:= Q(x)R(x) = \sqrt{1{-}10x{+}9x^2}
\end{align*}
\end{corollary}
\begin{proof}
Generating functions are derived by bisecting and simplifying the generating functions for the corresponding Motzkin triangle and Motzkin path sequences, given by Barcucci et al.~\cite{barcuccimotzkin}.
\end{proof}

\section{Extension to other lattice paths}
\label{sec:extensions}
Our approach can be used to explore \vc counterparts to other well known lattice paths, such as Dyck and Schr\"{o}der paths shown in Figure \ref{fig:summary}.

\subsection{Dyck Paths with Vertical Steps}
\label{sec:dyck}

Dyck paths are composed from the step vectors $\SVDS = \{\step{1,1}, \step{1,-1}\}$.  We now consider \vc lattice paths with the vector step set $\SVDW = \{\step{1,1}$, $\step{1,-1}$, $\step{0,1}$, $\step{0,-1} \}$ . The four types of \vc $\SVDW$ paths are presented in Table \ref{table:dyckcases}.

\begin{table}[H]
\begin{center}
\begin{tabular}{|l|l|l|} \hline
Name    & Restricted to Q & Leading Step        \\ \hline \hline
$\GDWSet$  & False        & $e_1 \in \SVDW$   \\ \hline
$\GDWRSet$ & False        & $e_1 \in \SVDS$   \\ \hline
$\DWSet$   & True         & $e_1 \in \SVDW$   \\ \hline
$\DWRSet$  & True         & $e_1 \in \SVDS$   \\ \hline
\end{tabular}
\end{center}
\caption{\VC lattice paths with vector step set $\SVDW$.}
\label{table:dyckcases}
\end{table}

\subsubsection{Recurrence Relations}
\label{sec:dyckrecurrence}

\begin{figure}[H]
\centering
    \def\svgwidth{0.3\textwidth}%
    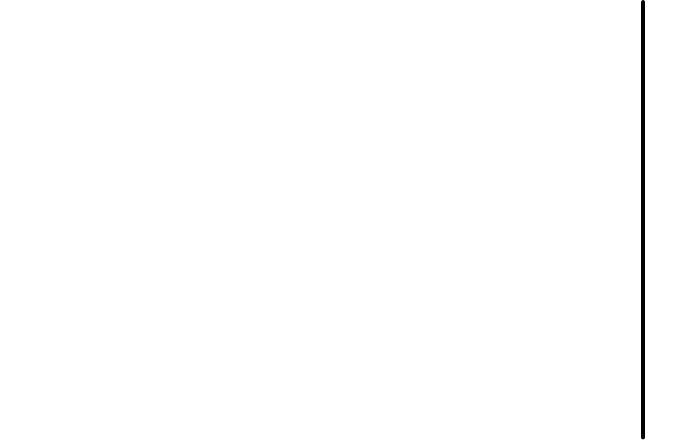%

\caption{All possible ways in which a \vc path using steps from the set $\SVDW$ can terminate at a lattice point (indicated by red dot).}
\label{fig:dyckrecursion}
\end{figure}

As before, we determine the recurrence relation by considering all ways the \vc $\SVDW$ paths can terminate at a lattice point.
\begin{lemma}
Let $\GDWCount(n,m)$ be the number of \vc paths in the half-plane created from step vectors $\SVDW$ and extending from $(0,0)$ to $(n,m)$.
The recurrence relation for $\GDWCount(n,m)$ is
\begin{eqnarray}
&\GDWCount(0,m) &= \I{m \in\{-1, 0, 1\}} \label{eq:gdw1} \\
&\GDWCount(n,m) &= \GDWCount(n{-}1,m{+}2)        + \GDWCount(n{-}1,m{+}1) + 2\GDWCount(n{-}1,m) \nonumber \\
&               &\quad{+}\, \GDWCount(n{-}1,m{-}1) + \GDWCount(n{-}1,m{-}2). \label{eq:gdw2}
\end{eqnarray}
\end{lemma}

For $\GDWRSet$, in which the first step is restricted to the set $\SVDS$, the initial condition is $\GDWRCount(0,m) = \I{m=0}$.  Equation \ref{eq:gdw2} can be rewritten for $\GDWRSet$ by substituting $\GDWRCount(a,b)$ for $\GDWCount(a,b)$.

The recurrence relation for $\DWSet$ and $\DWRSet$, paths restricted to the quarter-plane, is similar but with additional conditions concerning the first two rows of the triangle.
\begin{lemma}
Let $\DWCount(n,m)$ be the number of \vc paths in the quarter-plane created from step vectors $\SVDW$ and extending from $(0,0)$ to $(n,m)$.
The recurrence relation for $\DWCount(n,m)$ is
\begin{eqnarray}
&\DWCount(0,m) &= \I{m \in\{0, 1\}}, \nonumber \\[0.2cm]
\mathrlap{\text{otherwise, for $n>0$ and $m=0$ or $1$,}} \nonumber \\[0.2cm]
&\DWCount(n,0) &= \DWCount(n{-}1,2)     + \DWCount(n{-}1,1)     +  \DWCount(n{-}1,0),\label{eq:dw1} \\
&\DWCount(n,1) &= \DWCount(n{-}1,3)     + \DWCount(n{-}1,2)     + 2\DWCount(n{-}1,1) + \DWCount(n{-}1,0)\label{eq:dw2}\\[0.2cm]
\mathrlap{\text{ and for $n>0$, $m>1$,}} \nonumber \\[0.2cm]
&\DWCount(n,m) &= \DWCount(n{-}1,m{+}2) + \DWCount(n{-}1,m{+}1) + 2\DWCount(n{-}1,m) \nonumber \\
&       &\quad + \DWCount(n{-}1,m{-}1) + \DWCount(n{-}1,m{-}2). \label{eq:dw3}
\end{eqnarray}
\end{lemma}

For $\DWRSet$, in which the first step is restricted to the set $\SVDS$, the initial condition is $\DWRCount(0,m) = \I{m=0}$.  Equations \eqref{eq:dw1}, \eqref{eq:dw2} and \eqref{eq:dw3} can be rewritten for $\DWRSet$ by substituting $\DWRCount(a,b)$ for $\DWCount(a,b)$.

\subsubsection{An Explicit Bijection}
\label{sec:dyckbijection}

\begin{table}[H]
\begin{center}
\begin{tabular}{|l|l|l|} \hline
Family                                                          & Relationship                             & OEIS~\cite{oeis}       \\ \hline \hline
\multirow{2}{4.2cm}{Motzkin, no flat steps at even indices}        & $\DWCount(n,0)  = \MFlatCount(2n+1,0)$   & Bisection of A214938  \\
                                                                &                                          &                       \\ \hline
\multirow{2}{4.2cm}{Motzkin, no flat steps at odd indices}       & $\DWRCount(n,0) = \MFlatCount(2n,0)$     & Bisection of A214938  \\
                                                                &                                          &                       \\ \hline
\multirow{2}{4.2cm}{Grand Motzkin, no flat steps at even indices}     & $\GDWCount(n,0)  = \GMFlatCount(2n+1,0)$ & Bisection of A026520  \\
                                                                &                                          &                       \\ \hline
\multirow{2}{4.2cm}{Grand Motzkin, no flat steps at odd indices}    & $\GDWRCount(n,0) = \GMFlatCount(2n,0)$   & Bisection of  A026520 \\
                                                                &                                          &                       \\ \hline
\multirow{2}{4.2cm}{Grand Motzkin, no flat steps at even indices}     & $\GDWCount(n,m)  = \GMFlatCount(2n+1,m)$ & \multirow{2}{4.2cm}{Column bisection of A026519}  \\
                                                                &                                          &                       \\ \hline
\multirow{2}{4.2cm}{Grand Motzkin, no flat steps at odd indices}    & $\GDWRCount(n,m) = \GMFlatCount(2n,m)$   & \multirow{2}{4.2cm}{Column bisection of  A026519} \\
                                                                &                                          &                       \\ \hline
\end{tabular}
\end{center}
\caption{Correlation between \vc paths of type $\SVDW$ and subsets of the Motzkin family of paths.}
\label{table:dyckoeis}
\end{table}

The $\DWSet$ paths correspond to OEIS sequence A214938 which is the number of Motzkin paths in which flat steps do not occur at even indices. The $\GDWSet$ paths correspond to the OEIS table A026519 which is a prototypical sequence contributed by Kimberling and defined by a recurrence relation.  We will show that it corresponds to a subset of Grand Motzkin paths with the same constraint that flat steps cannot occur at even indices.

\begin{theorem}
There is a bijection between the subset of Motzkin (Grand Motzkin) paths terminating at $(2n+1,m)$ that avoid flat steps at \textbf{even} indices  and the \vc paths of type $\DWSet$ $(\GDWSet)$ terminating at $(n,m)$.
\label{theo:bijectionDyck}
\end{theorem}
\begin{proof}
This bijection can be described using the same $\phi$ and $\psi$ mappings outlined in Section \ref{sec:bijection}.
Start with a Motzkin path (or Grand Motzkin path) of length $2n+1$ in which the step vector $\step{1,0}$ occurs only at \textbf{odd} indices.  The application of $\phi$ to the first step, $e_1$, and subsequent steps of odd index, removes all horizontal steps during the substitution phase and inserts non-consecutive vertical steps. The resulting path contains only steps found in $\SVDW$ and is of type $\DWSet$ or $\GDWSet$.  Inversely, if we start with a $\DWSet$ or $\GDWSet$ path and apply $\psi$, all vertical steps are removed and horizontal steps, when added, occur only at odd indices producing a Motzkin path (or Grand Motzkin path) of the desired subset.
\end{proof}
\begin{theorem}
There is a bijection between the subset of Motzkin (Grand Motzkin) paths terminating at $(2n,m)$ that avoid flat steps at \textbf{odd} indices and the \vc paths of type $\DWSet$ $(\GDWSet)$ terminating at $(n,m)$.
\end{theorem}
\begin{proof}
The proof of this theorem follows from that of Theorem \ref{theo:bijectionDyck}.  By applying $\phi$ to the second step and alternating steps thereafter, the first step cannot be vertical.  Similarly, $\psi$ will only insert horizontal steps at even indices.
\end{proof}

It is important to note that if we reverse a Motzkin path of even length with horizontal steps not allowed at \textbf{even} indices we obtain a Motzkin path with horizontal steps not allowed at \textbf{odd} indices.

\begin{figure}[H]
\centering
    \def\svgwidth{\textwidth}%
    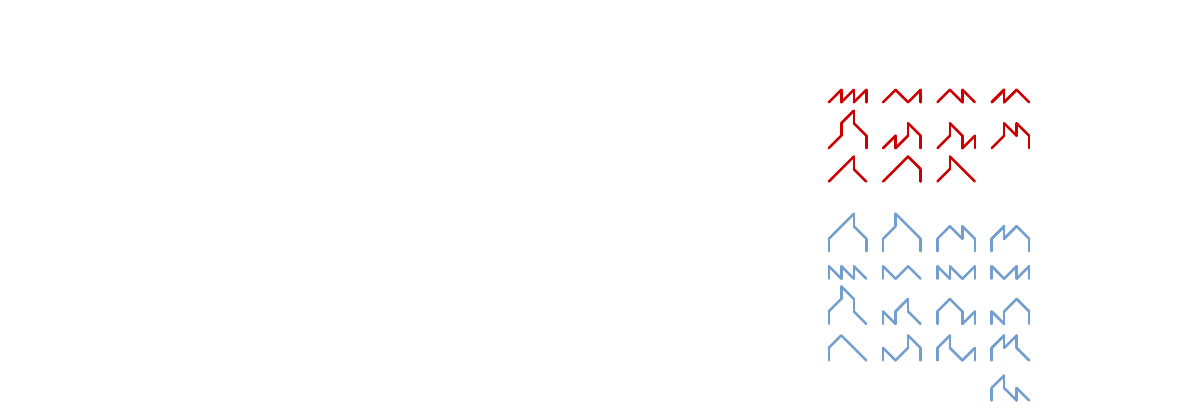%

\caption{Motzkin paths with flat steps not allowed at either odd or even indices and the corresponding $\DWRSet$ or $\DWSet$ paths respectively.}
\label{fig:exampleDyck}
\end{figure}

\subsubsection{Generating Functions}
\begin{theorem} \label{thm:GDW}
The bivariate generating functions for $\GDWRCount(n,m)$ and $\GDWCount(n,m)$ are
\begin{align*}
\gdwr(x,y) &:=  \sum_{n \ge 0} \sum_{m\in\mathbb{Z}} \GDWRCount(n,m) x^n y^m =  \frac{y^2}{y^2 {-} x\left(1{+}y^{2} \right)\left(1{+}y{+}y^{2} \right)},\\
\gdw(x,y)  &:=  \sum_{n \ge 0} \sum_{m\in\mathbb{Z}} \GDWCount(n,m) x^n y^m =  \frac{y\left(1{+}y{+}y^{2}\right)}{y^2 - x\left( 1 {+} y^2\right) \left(1{+}y{+}y^{2} \right)},\\
\gdwr(x,0) &:= \sum_{n \ge 0} \GDWRCount(n,0) x^n = \sqrt{\frac{2{-}7x{+}2\sqrt{1{-}8x{+}12x^2}}{4{-}31x{+}40x^2{+}12x^3}},\\
\intertext{and}
\gdw(x,0)  &:= \sum_{n \ge 0} \GDWCount(n,0) x^n = \sqrt{ \frac{2{-}4x{-}3x^2{-}2\sqrt{1{-}8x{+}12x^2}}{x\left(4{-}31x{+}40x^2{+}12x^3\right)}}
\end{align*}
\label{theo:gfdw}
\end{theorem}
\begin{proof}
The generating functions for $\gdwr(x,y)$ and $\gdw(x,y)$ were derived by expanding terms in the general summation~\cite[p.~5]{wilf} using the recurrence relations from Section \ref{sec:dyckrecurrence}.

To find $\gdwr(x,0)$, we isolate the terms of $\gdwr(x,y)$ that have no $y$ term. As $\gdwr(x,y)$ contains negative powers of $y$, one
cannot simply substitute $y=0$.  However, as $\gdwr(x,y)$ is a rational function we can divide by $y$ and compute the residues of the resulting function
in $y$ which approach the origin as $x$ approaches the origin:
\begin{align*}
\gdwr(x,0) = \sum_{i=1}^{n} Res\left(\frac{\gdwr(x,y)}{y}; \rho_i \right) 
\end{align*}
where ${\rho_1, ..., \rho_n}$ are the poles of $\frac{\gdwr(x,y)}{y}$ which approach the origin as $x$ approaches the origin.

Write $\gdwr(x,y)/y = G(x,y)/H(x,y)$ for co-prime polynomials $G$ and $H$.  There exist two singularities, roots $r_1$ and $r_2$ of $H(x,y)$, that are finite, and in fact vanish, when $x=0$:
\begin{align*}
r_1 &= \frac{   \sqrt{x(x{+}4)} {-} x {-} \sqrt{4x{-}14x^2{-}2x\sqrt{x(x{+}4)}}}{4x},\\
r_2 &= \frac{{-}\sqrt{x(x{+}4)} {-} x {+} \sqrt{4x{-}14x^2{+}2x\sqrt{x(x{+}4)}}}{4x}.
\end{align*}
The existence of these two roots follows from Proposition 6.1.8 of Stanley~\cite{stanley}, and their Puiseux expansions can determined implicitly using the Newton polygon of $H$. As $r_1$ and $r_2$ are simple poles of $H$,
\begin{align*}
Res\left( \frac{G(x,y)}{H(x,y)}; y=r_1\right) = \left.\frac{G(x,y)}{\frac{d{H(x,y)}}{d{y}}}\right|_{y=r_1} = \frac{\sqrt{2}}{\sqrt{4+x}\sqrt{2-7x-\sqrt{x(4+x)}}},\\
Res\left( \frac{G(x,y)}{H(x,y)}; y=r_2\right) = \left.\frac{G(x,y)}{\frac{d{H(x,y)}}{d{y}}}\right|_{y=r_2} = \frac{\sqrt{2}}{\sqrt{4+x}\sqrt{2-7x+\sqrt{x(4+x)}}}.
\end{align*}
The sum of these residues simplifies to give the final result, and the same approach can be used to find $\gdw(x,0)$.
\end{proof}

\begin{corollary}
The bivariate generating function for Grand Motzkin paths in which flat steps are not permitted at even indices is derived using the bisection relationship of Lemma \ref{theo:bijectionDyck}:
\begin{align*}
\gmflat(x,y) & = \gdwr(x^2,y)+x\gdw(x^2,y) = \frac{y^2 +xy(1+y+y^2)}{y^2-x^2(1+y^2)(1+y+y^2)}\\
\intertext{and}
\gmflat(x,0) & = \gdwr(x^2,0)+x\gdw(x^2,0) = \sqrt{ \frac{4{+}6x{-}11x^2{-}6x^3{-}3x^4{+}2x\sqrt{1{-}8x^2{+}12x^4}}{4{-}31x^2{+}40x^4{+}12x^6}}
\end{align*}
\end{corollary}

\begin{theorem}
\label{theo:dw}
The generating functions for $\DWRCount(n,m)$ and $\DWCount(n,m)$ are
\begin{align*}
\label{eq:dw4}
\dwr(x,y) & :=  \sum_{n \ge 0} \sum_{m \ge 0} \DWRCount(n,m) x^n y^m \\
          & = \frac{(r_1{-}y)(r_2{-}y)}{(1{-}r_1 r_2)(y^2{-}x(1 + y^2)(1 + y + y^2))}, \\
\dw(x,y) & :=  \sum_{n \ge 0} \sum_{m \ge 0} \DWCount(n,m) x^n y^m \\
         & = \frac{r_1 r_2 y{-}(1{+}y{+}y^2)\left(x(r_1+r_2){-}xy(1{-}r_1 r_2)\right)}{x(1{-}r_1 r_2)(y^2{-}x(1 + y^2)(1 + y + y^2))}, \\
\dwr(x,0) &:= \sum_{n\ge0}\DWRCount(n,0)x^n = \frac{{-}r_1r_2}{x(1{-}r_1r_2)},\\
\intertext{and}
\dw(x,0) &:= \sum_{n\ge0}\DWCount(n,0)x^n = \frac{r_1{+}r_2}{x(1{-}r_1r_2)}.
\end{align*}
\end{theorem}
\begin{proof}
Generating functions for $\DWSet$ and $\DWRSet$ paths (restricted to the quarter-plane) were determined using a combination of term expansion and the \emph{kernel method}~\cite{BanderierBousquet-MelouDeniseFlajoletGardyGouyou-Beauchamps2002}.  Term expansion in the general summation gives
\begin{align}
\dwr(x,y) &= \frac{y^2{-}x\left(1{+}y{+}y^{2}\right)\dwr(x,0) {-} xy S(x)}{y^2 {-} x\left( 1 {+} y^2\right) \left(1{+}y{+}y^{2} \right)},\\
\dw(x,y) &= \frac{y^2 {+} y^3 {-} x\left(1 {+} y {+} y^2\right)\dw(x,0) {-}xy T(x)}{y^2 {-} x\left( 1 {+} y^2\right)\left(1{+}y{+}y^{2} \right)},
\end{align}
where
\[
S(x) := \sum_{n\ge0}\DWRCount(n,1)x^n \quad \text{and} \quad
T(x) := \sum_{n\ge0}\DWCount(n,1)x^n.  \]

The term `kernel method' refers to a collection of techniques for solving functional equations often appearing in lattice path problems. Clearing fractions in Equation~\eqref{eq:dw4} implies
\begin{equation} K(x,y) \dwr(x,y) = y^2{-}x\left(1{+}y{+}y^{2}\right)\dwr(x,0) {-} xy S(x), \label{eq:kernel} \end{equation}
where $K(x,y)$ is the \emph{kernel}
\[ K(x,y) = y^2 - x\left(1+y^2\right)\left(1+y+y^{2}\right). \]
As discussed in the proof of Theorem~\ref{thm:GDW}, there are two roots $r_1$ and $r_2$ of $K$ which vanish at $x=0$. Substituting $y=r_1$ and $y=r_2$ into Equation~\eqref{eq:kernel} gives two equations in two unknowns $\dwr(x,0)$ and $S(x)$, which can be solved explicitly.
The same approach can be used to find $\dw(x,y)$ and $\dw(x,0)$.
\end{proof}

The bijection with Motzkin paths in which flat steps are not permitted at even indices can then be used to obtain generating functions for this subset of Motzkin paths both in the general triangle case and for paths terminating on the $x$-axis.

\begin{theorem}
The univariate and bivariate generating functions for Motzkin paths terminating on the $x$-axis in which flat steps are not permitted at odd indices are
\begin{align*}
\mflat(x,y)  &:= \sum_{n \ge 0} \sum_{m \ge 0}\MFlatCount(n,0) x^n y^m\\
& = \frac{x (q_1 {-} y) (q_2 {-} y) + q_1 q_2 y - x^2 (q_1 {+} q_2 {-} y {+} q_1 q_2 y) (1 {+} y {+} y^2)}{x(1 {-} q_1 q_2) (y^2 {-} x^2 (1 {+} y^2) (1 {+} y {+} y^2))},
\intertext{and}
\mflat(x,0) &:= \sum_{n\ge0}\MFlatCount(n,0)x^n = \frac{x(q_1{+}q_2){-}q_1 q_2}{x^2(1{-}q_1 q_2)},\\
\intertext{where}
q_1 &:= -\frac{x - \sqrt{4 + x^2} + \sqrt{4{-}14x^2{-}2x\sqrt{4 + x^2}}}{4 x},\\
q_2 &:= -\frac{x + \sqrt{4 + x^2} - \sqrt{4{-}14x^2{+}2x\sqrt{4 + x^2}}}{4 x}.
\end{align*}
\label{theo:flat}
\end{theorem}
\begin{proof}
The generating function $\mflat(x,y)$ is derived, as a consequence of the bijection relationship, by combining $\dwr(x,y)$ and $\dw(x,y)$. By setting $y=0$ in $\mflat(x,y)$, we obtain $\mflat(x,0)$.
\end{proof}
\subsection{Schr\"{o}der and Delannoy Paths with Vertical Steps}
\label{sec:schroder}

Delannoy paths are a composition of the step vectors $\SVCS = \{\step{1,1}$, $\step{1,-1}$, $\step{2,0}\}$.\footnote{Note: we take a non-traditional orientation of the Schr\"{o}der and Delannoy paths. Typically they are drawn with steps $\{\step{1,1}$, $\step{1,0}$ and $\step{0,1}\}$ travelling from the south-west to north-east corners of an $n\times n$ square, a $45^{\circ}$ rotation from our representation.}
Schr\"{o}der paths use the same step vectors but are restricted to the first quadrant.
Our \vc variant is created from the step set $\SVCW = \{\step{1,1}$, $\step{1,-1}$, $\step{2,0}$, $\step{0,2}$, $\step{0,-2}\}$ where the height of the vertical step vector is equal to the width of the horizontal step vector.  Table \ref{table:schrodercases} outlines the four cases.

\begin{table}[H]
\begin{center}
\begin{tabular}{|l|l|l|} \hline
Name  &Restricted to Q& Leading Step   \\ \hline \hline
$\GCWSet$  & False & $e_1 \in \SVCW$   \\ \hline
$\GCWRSet$ & False & $e_1 \in \SVCS$   \\ \hline
$\CWSet$   & True  & $e_1 \in \SVCW$   \\ \hline
$\CWRSet$  & True  & $e_1 \in \SVCS$   \\ \hline
\end{tabular}
\end{center}
\caption{\VC lattice paths with vector step set $\SVCW$.}
\label{table:schrodercases}
\end{table}

\subsubsection{Recurrence Relations}
\label{sec:schroderrecurrence}
\begin{figure}[H]
\centering
    \def\svgwidth{0.4\textwidth}%
    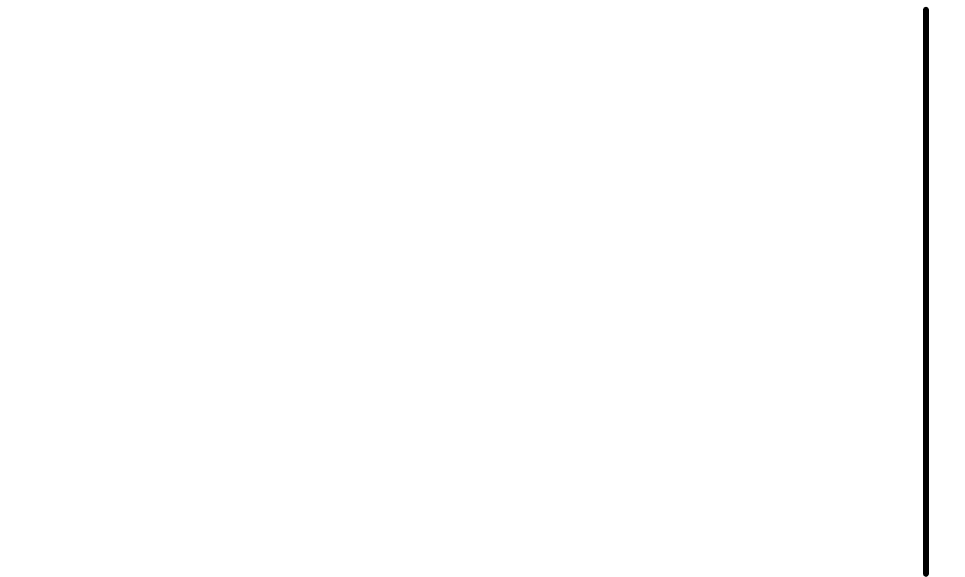%

\caption{All possible ways a \vc path can terminate at a lattice point (red dot) using step set $\SVCW$.}
\label{fig:schroderrecursion}
\end{figure}

Once again, the recurrence relation for $\GCWSet$ paths is determined by considering all ways in which the path can terminate at a lattice point as shown in Figure \ref{fig:schroderrecursion}.
\begin{lemma}
The recurrence relation for $\GCWCount(n,m)$ is
\begin{eqnarray}
&\GCWCount(0,m) &= \I{m\in\{-2, 0, 2\}} \text{ if $n=0$, otherwise}\nonumber \\
&\GCWCount(n,m) &= \GCWCount(n{-}1,m{+}3) + \GCWCount(n{-}2,m{+}2) + 2\GCWCount(n{-}1,m{+}1) + \GCWCount(n{-}2,m) \nonumber \\
&       &\quad{+}\, 2\GCWCount(n{-}1,m{-}1) + \GCWCount(n{-}2,m{-}2) + \GCWCount(n{-}1,m{-}3). \nonumber
\end{eqnarray}
\end{lemma}

The recurrence relation for $\CWSet$ (restricted to the quarter-plane) is very similar but with additional conditions concerning steps near the $x$-axis.
\begin{lemma}
The recurrence relation for $\CWCount(n,m)$ is
\begin{eqnarray}
&\CWCount(p,m) &= 0 \text{ if } p<0 \nonumber \\
&\CWCount(0,m) &= \I{m \in\{0, 2\}}\nonumber \\
&\CWCount(n,0) &= \CWCount(n{-}1,3)    + \CWCount(n{-}2,2)   + 2\CWCount(n{-}1,1)   + \CWCount(n{-}2,0)\nonumber \\
&\CWCount(n,1) &= \CWCount(n{-}1,4)    + \CWCount(n{-}2,3)   + 2\CWCount(n{-}1,2)   + \CWCount(n{-}2,1) +  \CWCount(n{-}1,0)\nonumber \\
&\CWCount(n,2) &= \CWCount(n{-}1,5)    + \CWCount(n{-}2,4)   + 2\CWCount(n{-}1,3)   + \CWCount(n{-}2,2) + 2\CWCount(n{-}1,1) \nonumber \\
&              &\;\; +\, \CWCount(n{-}2,0)\nonumber \\
&\CWCount(n,m) &= \CWCount(n{-}1,m+3)  + \CWCount(n{-}2,m+2) + 2\CWCount(n{-}1,m+1) + \CWCount(n{-}2,m) \nonumber \\
&        &\quad{+}\, 2\CWCount(n{-}1,m-1)+ \CWCount(n{-}2,m-2) +  \CWCount(n{-}1,m-3).\nonumber
\end{eqnarray}
\end{lemma}

\subsubsection{An Explicit Bijection}
The mappings $\phi$ and $\psi$ do not apply to the \vc variants of the Schr\"{o}der-Delannoy paths.  In $\SVCW$, the $\rightarrow$ step vector has twice the horizontal displacement of the $\nearrow$ and $\searrow$ step vectors, a property that causes substitutions of the type used in $\phi$ and $\psi$ to produce paths of varying length.  The OEIS does not contain entries for objects equinumerous with members of the $\SVCW$ path family.

There does, however, exist a bijection to a subset of Schr\"{o}der-Delannoy paths consisting of paths that are smooth at every third column. By \emph{smooth} at a column $c_i$, we mean that the path does not change direction at the integer $x$-position $c_i$.  This can occur either because the edges before and after $x=c_i$ have the same step vector, or, because $x=c_i$ intersects a midpoint of an edge that spans multiple columns.  Examples of these paths are shown in Figure \ref{fig:exampleSchroder}.  For conciseness, we shall refer to this subset of Schr\"{o}der/Delannoy paths as \ssd paths.

We define a mapping $\gamma$, illustrated in Figure \ref{fig:schrmap}, in which the following substitutions are performed on a \ssd path to either a single step or a pair of steps centered at a regular interval of every 3 columns on the integer lattice:
\begin{eqnarray*}
   \gamma :& \step{1,1}\step{1,1}  & \to  \step{0,2}\\
   \gamma :& \step{1,-1}\step{1,-1} & \to  \step{0,-2}\\
   \gamma :& \step{2,0}\step{2,0} & \to  \step{2,0}\\
   \gamma :& \step{2,0}  & \to  \text{ remove step}\\
\end{eqnarray*}

\begin{figure}[H]
\centering
    \def\svgwidth{\textwidth}%
    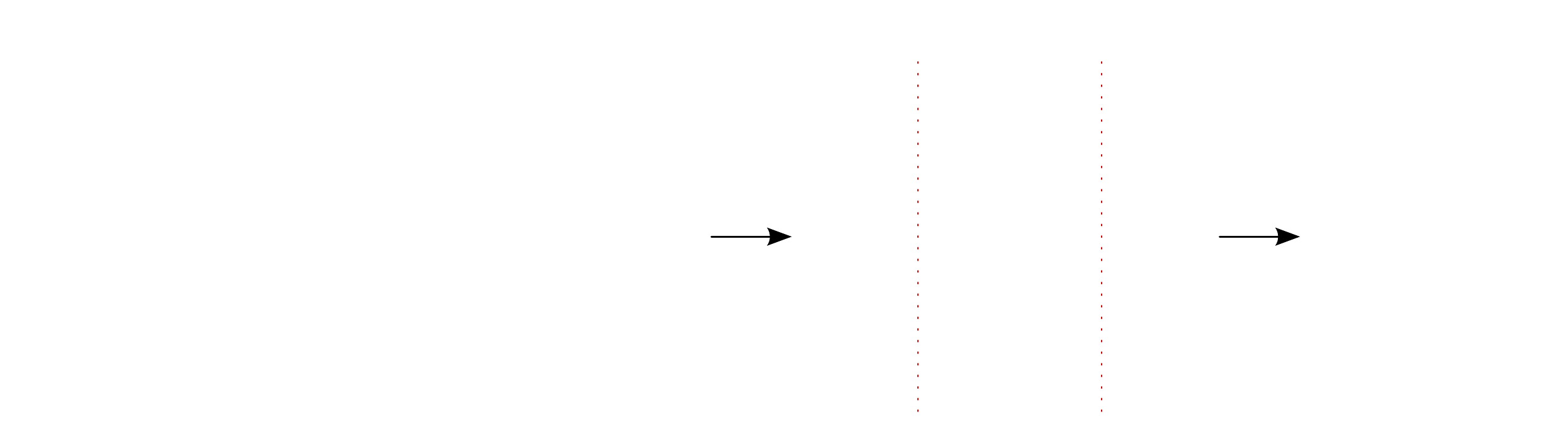%

\caption{Bijective mapping from a Schr\"{o}der path (length $=18$) that is smooth at horizontal distances $3i+2$ from the origin (indicated by red dashed lines) for $0 \le i < 6$ to a $\CWRSet$ path (length $=6$).}
\label{fig:schrmap}
\end{figure}

The mapping $\gamma$ is similar to $\phi$ but maintains a constant horizontal width by replacing vertical steps in the $\SVCW$ paths with \emph{two} diagonal steps and replacing horizontal steps in $\SVCW$ with \emph{two} horizontal steps.

We define the inverse mapping $\lambda$, illustrated in Figure \ref{fig:schrunmap}, in which the following substitutions are performed on steps that are centered at each column in the integer lattice:
\begin{eqnarray*}
\lambda : & \step{0,2}  & \to \step{1,1}\step{1,1}\\
\lambda : & \step{0,-2} & \to \step{1,-1}\step{1,-1} \\
\lambda : & \step{2,0}  & \to \step{2,0}\step{2,0} \\
\lambda : & \text{no step}  & \to \text{insert }\step{2,0}
\end{eqnarray*}

\begin{figure}[H]
\centering
    \def\svgwidth{\textwidth}%
    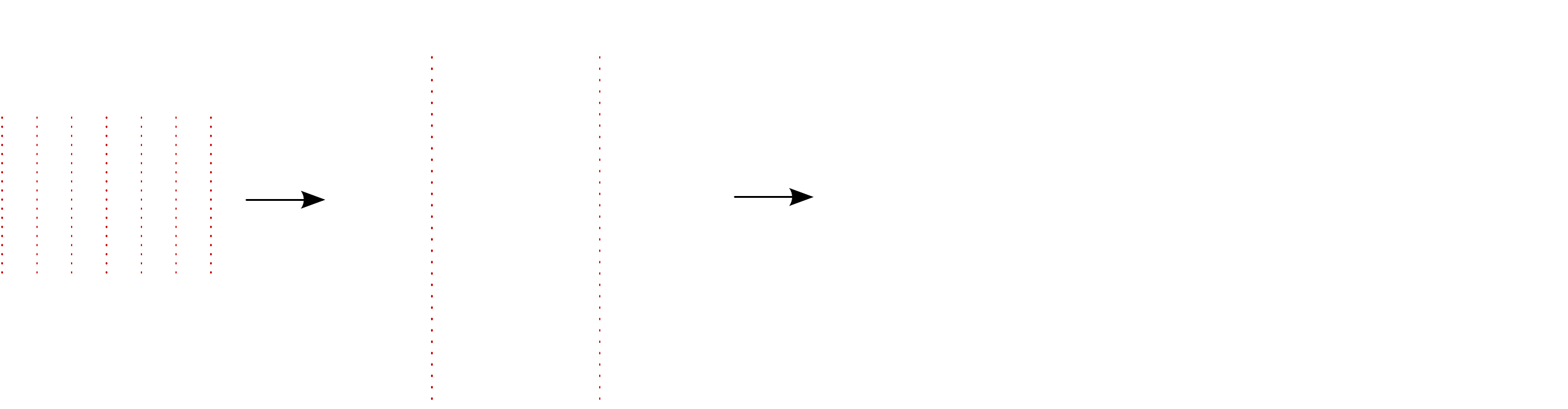%

\caption{Bijective mapping from a $\GCWSet$ path (length $=6$) to a Delannoy path (length $=20$) that is smooth at  horizontal distances $3i+1$ from the origin (indicated by red dashed lines) for $0 \le i \le 6$.}
\label{fig:schrunmap}
\end{figure}

\begin{theorem}The mapping $\gamma$ is a bijection from Schr\"{o}der (Delannoy) paths terminating at $(3n+2,m)$ that are smooth at each horizontal distance $3i+1$ from the origin for $0 \le i \le n$ to \vc paths of type  $\CWSet$ $(\GCWSet)$ terminating at $(n,m)$.
\label{theo:bijectionSchrod}
\end{theorem}
\begin{proof}
We start by proving that $\gamma$ applied to a \ssd Schr\"{o}der (Delannoy) path always produces a \vc $\CWSet$ $(\GCWSet)$ path.  The set of step vectors $\SVCS$ is a subset of $\SVCW$ and all substitutions in $\gamma$ are from the set  $\SVCW$, therefore the resulting path is made up of steps from $\SVCW$.  Because there is a spacing of three columns in the pre-image between $\gamma$ substitutions, there is at least one column between vertical steps in the image satisfying the non-consecutive condition of \vc  $\SVCW$ paths.  Each $\gamma$ substitution reduces the horizontal path length by 2 resulting, after $n+1$ substitutions, in a path of length $n$.  As in the Motzkin path bijections of Section \ref{sec:bijection}, there is no change in the height of step endvertices ensuring path images remain in the same half or quadrant as their pre-image.

Next we will prove that the \vc $\SVCW$ paths and  \ssd Schr\"{o}der (Delannoy) paths are equinumerous.  There are exactly four ways in which a Schr\"{o}der path can be smooth across a column: two repeated steps of the same type $(\rightarrow\rightarrow$, $\nearrow\nearrow$, and $\searrow\searrow)$ or a single horizontal step of length two.  All four of these cases are uniquely mapped by $\gamma$ to steps in $\SVCW$.  Since each \ssd Schr\"{o}der (Delannoy) path under $\gamma$ produces a unique \vc $\SVCW$ path, the number of \vc $\SVCW$ paths is greater than or equal to the number of \ssd Schr\"{o}der (Delannoy) paths.  Similarly, the mapping $\lambda$ applied to a \vc $\SVCW$ path can be proven to always produce a unique \ssd Schr\"{o}der (Delannoy) path. From this we can conclude that the number of \ssd Schr\"{o}der (Delannoy) paths is greater than or equal to the number of \vc $\SVCW$ paths.  The two contradictory inequalities imply that the two sets of paths are equinumerous.
\end{proof}

\begin{theorem}The mapping $\gamma$ is a bijection from Schr\"{o}der (Delannoy) paths terminating at $(3n,m)$ that are smooth at a horizontal distance $3i+2$ from the origin for $0 \le i < n$ to \vc paths of type  $\CWRSet$ $(\GCWRSet)$ terminating at $(n,m)$.
\label{theo:bijectionSchrod2}
\end{theorem}
\begin{proof}
The proof is similar to that of Theorem \ref{theo:bijectionSchrod}.  The first step of the \ssd Schr\"{o}der (Delannoy) path is not modified by $\gamma$, therefore there can be no leading vertical step.  There are $n$ substitutions, resulting in a \vc path of length $n$.
\end{proof}

\begin{figure}[H]
\centering
    \def\svgwidth{0.8\textwidth}%
    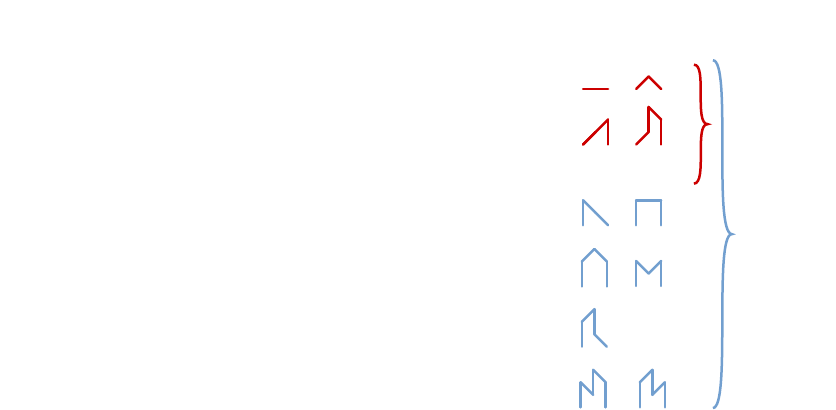%

\caption{Schr\"{o}der paths ($n=6$ and $n=8$) that are smooth at a regular interval of 3 columns and the corresponding $\CWRSet$ and $\CWSet$ paths ($n=2$).}
\label{fig:exampleSchroder}
\end{figure}

\subsubsection{Generating Functions}
\begin{theorem}
The generating functions for $\GCWRCount(n,m)$ and $\GCWCount(n,m)$ are
\begin{align*}
\gcwr(x,y) &:= \sum_{n \ge 0} \sum_{m\in\mathbb{Z}} \GCWRCount(n,m) x^n y^m = \frac{y^3}{y^3 {-} x\left(1{+}xy{+}y^2\right)\left(1{+}y^2{+}y^4 \right)},\\
\gcw(x,y)  &:= \sum_{n \ge 0} \sum_{m\in\mathbb{Z}} \GCWCount(n,m) x^n y^m = \frac{y\left( 1 {+} y^2 {+} y^4 \right)}{y^3 {-} x\left( 1{+}xy{+}y^2 \right)\left( 1 {+} y^2 {+} y^4 \right)},\\
\gcwr(x,0) &:= \sum_{n \ge 0} \GCWRCount(n,0) x^n = \frac{\gdflat(x,0)+\gdflat(-x,0)}{2},\\
\gcw(x,0) &:= \sum_{n \ge 0} \GCWCount(n,0) x^n = \frac{\gdflat(x,0)-\gdflat(-x,0)}{2x},
\intertext{where}
\gdflat(x,0)&= \sum_{y\in\{s_1,s_2,s_3\}}\frac{y^2+x \left(1+y^2+y^4\right)}{3 y^2-2 x y \left(2+4 y^2+3 y^4\right)-x^2 \left(1+3 y^2+5 y^4\right)},
\intertext{and}
s_1,s_2,s_3&:= \text{unique roots of }y^3 {-} x\left(1{+}xy{+}y^2\right)\left(1{+}y^2{+}y^4\right)\text{ that are finite at 0}.
\end{align*}
\end{theorem}
\begin{proof}
The generating functions for $\GCWSet$ and $\GCWRSet$ paths, expressed as $\gcw(x,y)$ and $\gcwr(x,y)$, can be determined by expanding terms in the general summation using the recurrence relations found in Section~\ref{sec:schroderrecurrence}.
To enumerate only the paths that terminate on the x-axis, we first combine the general equations, each of which has zero values at alternating cell positions in the triangle representation, to obtain $\gdflat(x,y)$, an enumeration of Delannoy paths that are smooth at regular horizontal distances,
\begin{align*}
\gdflat(x,y)&:= \gcwr(x,y)+x\gcw(x,y) = \frac{y^3+xy(1+y^2+y^4)}{y^3 {-} x\left(1{+}xy{+}y^2\right)\left(1{+}y^2{+}y^4 \right)}.
\end{align*}
Because $\gdflat(x,y)$ is a ratio of polynomials, we can once again apply the calculus of residues to get an explicit formula for $\gdflat(x,0)$.
Proposition 6.1.8 of Stanley~\cite{stanley} implies there are three roots of the kernel, $H:=y^3 {-} x\left(1{+}xy{+}y^2\right)\left(1{+}y^2{+}y^4\right)$, that are finite as $x\rightarrow 0$, and they all vanish at the origin. We sum over the residues evaluated at these three roots to obtain an exact form for $\gdflat(x,0)$. Finally, we use bisection to return to our original goal.
\end{proof}

\begin{theorem}
The generating functions for $\CWRCount(n,m)$ and $\CWCount(n,m)$ are
\begin{align*}
\cwr(x,y) & := \sum_{n \ge 0} \sum_{m \ge 0} \CWRCount(n,m) x^n y^m \\
&= \frac{-\left(s_1-y\right)\left(s_2-y\right)\left(s_3-y\right)\left(1-s_1 s_2 s_3 y\right)}
{(1+s_1 s_2 s_3 (s_1+s_2+s_3)) \left(y^3 {-} x\left( 1 {+} xy {+} y^2\right)\left( 1 {+} y^2 {+} y^4 \right)\right)}, \\
\cw(x,y) & := \sum_{n \ge 0} \sum_{m \ge 0} \CWCount(n,m) x^n y^m \\
& = \frac{-(s_1-y)(s_2-y)(s_3-y)}{(1+s_1 s_2 s_3 (s_1+s_2+s_3)) \left(y^3 {-} x\left( 1 {+} xy {+} y^2\right)\left( 1 {+} y^2 {+} y^4 \right)\right)}\\
&\quad\times\Bigg(1{+}\frac{1}{2}\left((s_1{+}s_2)^2{+}(s_1{+}s_3)^2{+}(s_2{+}s_3)^2\right)\\
&\qquad + \left(s_1+s_2+s_3+s_1 s_2 s_3\left(s_1 s_2 +s_1 s_3+s_2 s_3 -1\right)\right) y \\
&\qquad + \left(1+s_1 s_2 s_3\left(s_1 + s_2 + s_3\right)\right) y^2\Bigg),\\
\cwr(x,0) & := \sum_{n \ge 0}\CWRCount(n,0) x^n,\\
          & = \frac{s_1s_2s_3}{x(1+s_1 s_2 s_3(s_1+s_2+s_3))}\\
\cw(x,0) & :=\sum_{n \ge 0}\CWCount(n,i) x^n, \\
         & = \frac{s_1s_2s_3(1{+}\frac{1}{2}((s_1{+}s_2)^2{+}(s_1{+}s_3)^2{+}(s_2{+}s_3)^2))}
                  {x(1+s_1 s_2 s_3(s_1+s_2+s_3))}
\end{align*}
\label{theo:gfcwr}
\end{theorem}
\begin{proof}
Again we use a combination of term expansion and the kernel method.  By term expansion in the general summation, we obtain
\begin{align}
\label{eq:cw1}
\cwr(x,y) & = \frac{y^3 {-} x\left(1{+}xy{+}2y^{2}{+}y^{4}\right)\cwr(x,0) {-} xy(1{+}xy)W_1(x) {-} xy^2W_2(x)}{y^3 {-} x\left( 1 {+} xy {+} y^2\right)\left( 1 {+} y^2 {+} y^4 \right)}\\
\cw(x,y) & = \frac{y\left( y^2{+}y^4\right) {-} x\left(1{+}xy{+}2y^{2}{+}y^{4}\right)\cw(x,0) {-} xy(1{+}xy)V_1(x) {-} xy^2V_2(x)}{y^3 {-} x\left( 1 {+} xy {+} y^2\right)\left( 1 {+} y^2 {+} y^4 \right)}, \\
\intertext{where}
W_i(x) & := \sum_{n \ge 0}\CWRCount(n,i) x^n, \quad i \in {1,2},  \nonumber\\
V_i(x) & := \sum_{n \ge 0}\CWCount(n,i) x^n, \quad i \in {1,2}.  \nonumber
\end{align}
As in the half-plane case, use the three roots $s_1,s_2$ and $s_3$ of the kernel $K(x,y) = y^3 {-} x\left( 1 {+} xy {+} y^2\right)\left( 1 {+} y^2 {+} y^4 \right)$, which vanish at the origin, to obtain a system of three equations in the three unknowns $\cwr(x,0)$, $W_1(x)$ and $W_2(x)$.  Solving this system gives the stated result.
\end{proof}

\section{Future Work}
Inspired by a lattice path model arising in the study of lace, we have studied the \vc extensions of Motzkin, Dyck, Schr\"{o}der and Delannoy paths.  There are many potential generalizations: the same approach can be applied to any set of directed step vectors; we have focused our discussion on vertical steps that are of the same length as the horizontal steps but the length of the vertical steps can also be allowed to vary; the restriction that there can be zero consecutive vertical steps could be generalized to allow at most $k$ consecutive vertical steps.
We have demonstrated the existence of a bijective relationship between \vc paths and subsets of lattice path families with similar step vectors.  Because the vertical constraint has clear geometric consequences, it is straight-forward to derive a recurrence relation.  Consequently, \vc paths may serve as a useful tool for analyzing subsets of other lattice path families subject to periodic constraints.

\bibliographystyle{plain}
\bibliography{mybib}

\pagebreak
\appendix

\begin{landscape}
\section{Appendix: Integer Sequence Tables}
\label{sec:tables}

\begin{table}[H]
\centering
\begin{adjustbox}{max width=\textwidth}
\begin{tabular}{|r|r|r|r|r|r|r|r|r|r|r|r|r|r|r|r|r|r|}
  \hline
  $n\backslash m$ &0 &$\abs{1}$ &$\abs{2}$ &$\abs{3}$ &$\abs{4}$ &$\abs{5}$ &$\abs{6}$ &$\abs{7}$ &$\abs{8}$ &$\abs{9}$ &$\abs{10}$ &$\abs{11}$ &$\abs{12}$ &$\abs{13}$ &$\abs{14}$ &$\abs{15}$ &$\abs{16}$ \\ \hline \hline
0                 &1 & & & & & & & & & & & & & & & & \\		
1                 &3 &2 &1  & & & & & & & & & & & & & &\\		
2                 &19 &16 &10 &4 &1 & & & & & & & & & & &   &\\
3                 &141 &126 &90 &50 &21 &6 &1 & & & & & & & & &  &\\
4                 &1107 &1016 &784 &504 &266 &112 &36 &8 &1 & & & & & & & &\\
5                 &8953 &8350 &6765 &4740 &2850 &1452 &615 &210 &55 &10 &1 & & & & & &\\
6                 &73789 &69576 &58278 &43252 &28314 &16236 &8074 &3432 &1221 &352 &78 &12 &1 & & & &\\
7                 &616227 &585690 &502593 &388752 &270270 &168168 &93093 &45474 &19383 &7098 &2184 &546 &105 &14 &1 & &\\
8                 &5196627  &4969152 &4343160 &3465840 &2520336  &1665456 &996216  &536640 &258570 &110448 &41328 &13328 &3620 &800 &136 &16 &1 \\ \hline
OEIS              &A082758 & & & & & & & & & & & & & & & & \\	\hline
\end{tabular}
\end{adjustbox}
\caption{Number of paths terminating at point (n,m) for $\GMWRSet$}
\label{table:recurrence1}
\end{table}

\begin{table}[H]
\centering
\begin{adjustbox}{max width=\textwidth}
\begin{tabular}{|r|r|r|r|r|r|r|r|r|r|r|r|r|r|r|r|r|r|r|}
  \hline
  $n\backslash m$ &0 &$\abs{1}$ &$\abs{2}$ &$\abs{3}$ &$\abs{4}$ &$\abs{5}$ &$\abs{6}$ &$\abs{7}$ &$\abs{8}$ &$\abs{9}$ &$\abs{10}$ &$\abs{11}$ &$\abs{12}$ &$\abs{13}$ &$\abs{14}$ & $\abs{15}$ &$\abs{16}$ &$\abs{17}$ \\ \hline \hline
0                 &1  &1 & & & & & & & & & & & & & & & &\\	
1                 &7  &6  &3  &1 & & & & & & & & & & & & &  &\\
2                 &51  &45  &30  &15  &5  &1 & & & & & & & & & & &  &\\
3                 &393  &357  &266  &161  &77  &28  &7  &1 & & & & & & & & &  &\\
4                 &3139  &2907  &2304  &1554  &882  &414  &156  &45  &9  &1  & & & & & & &  &\\
5                 &25653  &24068  &19855  &14355  &9042  &4917  &2277  &880  &275  &66  &11  &1 & & & & &  &\\
6                 &212941  &201643  &171106  &129844  &87802  &52624  &27742  &12727  &5005  &1651  &442  &91  &13  &1 & & &  &\\
7                 &1787607  &1704510  &1477035  &1161615  &827190  &531531  &306735  &157950  &71955  &28665  &9828  &2835  &665  &120  &15  &1 &  &\\
8                 &15134931  &14508939  &12778152  &10329336  &7651632  &5182008  &3198312  &1791426  &905658  &410346  &165104  &58276  &17748  &4556  &952  &153  &17  &1\\ \hline
OEIS              &Bisect A002426 & & & & & & & & & & & & & & & & & \\ \hline
\end{tabular}
\end{adjustbox}
\caption{Number of paths terminating at point (n,m) for $\GMWSet$}
\label{table:recurrence2}
\end{table}

\begin{table}[H]
\centering
\begin{adjustbox}{max width=\textwidth}
\begin{tabular}{|r||r|r|r|r|r|r|r|r|r|r|r|r|r|r|r|r|r|}
  \hline
  $n\backslash m$ &0 &1 &2 &3 &4 &5 &6 &7 &8 &9 &10 &11 &12 &13 &14 & 15 &16 \\ \hline \hline
0                 &1  & & & & & & & & & & & & & & & & \\
1                 &2  &2  &1  & & & & & & & & & & & & & & \\
2                 &9  &12  &9  &4  &1  & & & & & & & & & & & & \\
3                 &51  &76  &69  &44  &20  &6  &1  & & & & & & & & & & \\
4                 &323  &512  &518  &392  &230  &104  &35  &8  &1  & & & & & & & & \\
5                 &2188  &3610  &3915  &3288  &2235  &1242  &560  &200  &54  &10  &1  & & & & & & \\
6                 &15511  &26324  &29964  &27016  &20240  &12804  &6853  &3080  &1143  &340  &77  &12  &1 & & & & \\
7                 &113634  &196938  &232323  &220584  &177177  &122694  &73710  &38376  &17199  &6552  &2079  &532  &104  &14  &1  & & \\
8                 &853467  &1503312  &1822824  &1800384  &1524120  &1128816  &737646  &426192  &217242  &97120  &37708  &12528  &3484  &784  &135  &16  &1 \\ \hline
OEIS              &A026945 & & & & & & & & & & & & & & & & \\	\hline
\end{tabular}
\end{adjustbox}
\caption{Number of paths terminating at point (n,m) for $\MWRSet$}
\label{table:recurrence3}
\end{table}

\begin{table}[H]
\centering
\begin{adjustbox}{max width=\textwidth}
\begin{tabular}{|r||r|r|r|r|r|r|r|r|r|r|r|r|r|r|r|r|r|r|}
  \hline
  $n\backslash m$ &0 &1 &2 &3 &4 &5 &6 &7 &8 &9 &10 &11 &12 &13 &14 & 15 &16 &17\\ \hline \hline
0                 &1  &1  & & & & & & & & & & & & & & & &\\
1                 &4  &5  &3  &1  & & & & & & & & & & & & & &\\
2                 &21  &30  &25  &14  &5  &1  & & & & & & & & & & & &\\
3                 &127  &196  &189  &133  &70  &27  &7  &1  & & & & & & & & & &\\
4                 &835  &1353  &1422  &1140  &726  &369  &147  &44  &9  &1  & & & & & & & &\\
5                 &5798  &9713  &10813  &9438  &6765  &4037  &2002  &814  &264  &65  &11  &1  & & & & & &\\
6                 &41835  &71799  &83304  &77220  &60060  &39897  &22737  &11076  &4563  &1560  &429  &90  &13  &1  & & & &\\
7                 &310572  &542895  &649845  &630084  &520455  &373581  &234780  &129285  &62127  &25830  &9163  &2715  &650  &119  &15  &1 & & \\
8                 &2356779  &4179603  &5126520  &5147328  &4453320  &3390582  &2292654  &1381080  &740554  &352070  &147356  &53720  &16796  &4403  &935  &152  &17  &1  \\ \hline
OEIS              &A099250 & & & & & & & & & & & & & & & & &\\	\hline
\end{tabular}
\end{adjustbox}
\caption{Number of paths terminating at point (n,m) for $\MWSet$}
\label{table:recurrence4}
\end{table}

\begin{table}[H]
\centering
\begin{adjustbox}{max width=\textwidth}
\begin{tabular}{|r||r|r|r|r|r|r|r|r|r|r|r|r|r|r|r|r|r|r|}
  \hline
  $n\backslash m$ &0 &$\abs{1}$ &$\abs{2}$ &$\abs{3}$ &$\abs{4}$ &$\abs{5}$ &$\abs{6}$ &$\abs{7}$ &$\abs{8}$ &$\abs{9}$ &$\abs{10}$ &$\abs{11}$ &$\abs{12}$ &$\abs{13}$ &$\abs{14}$ & $\abs{15}$ &$\abs{16}$ &$\abs{17}$ \\ \hline \hline
0    &                    1&        &        &        &        &        &        &        &        &        &        &        &        &        &        &        &        &                \\
1    &                    2&       1&       1&        &        &        &        &        &        &        &        &        &        &        &        &        &        &                \\
2    &                    8&       6&       5&       2&       1&        &        &        &        &        &        &        &        &        &        &        &        &                \\
3    &                   38&      33&      27&      16&       9&       3&       1&        &        &        &        &        &        &        &        &        &        &                \\
4    &                  196&     180&     150&     104&      65&      32&      14&       4&       1&        &        &        &        &        &        &        &        &                \\
5    &                 1052&     990&     845&     635&     430&     251&     130&      55&      20&       5&       1&        &        &        &        &        &        &                \\
6    &                 5774&    5502&    4797&    3786&    2721&    1752&    1016&     516&     231&      86&      27&       6&       1&        &        &        &        &                \\
7    &                32146&   30863&   27377&   22344&   16793&   11543&    7252&    4117&    2107&     952&     378&     126&      35&       7&       1&        &        &                \\
8    &               180772&  174456&  156900&  131264&  102102&   73592&   49064&   30088&   16913&    8632&    3976&    1624&     582&     176&      44&       8&       1&                \\ \hline
OEIS&              A026520& A026521 & A026522 & A026523 & A026524 & & & & & & & & & & & & & \\	\hline
\end{tabular}
\end{adjustbox}
\caption{Number of paths terminating at point (n,m) for $\GDWRSet$}
\label{table:recurrence6}
\end{table}

\begin{table}[H]
\centering
\begin{adjustbox}{max width=\textwidth}
\begin{tabular}{|r||r|r|r|r|r|r|r|r|r|r|r|r|r|r|r|r|r|r|}
  \hline
  $n\backslash m$ &0 &$\abs{1}$ &$\abs{2}$ &$\abs{3}$ &$\abs{4}$ &$\abs{5}$ &$\abs{6}$ &$\abs{7}$ &$\abs{8}$ &$\abs{9}$ &$\abs{10}$ &$\abs{11}$ &$\abs{12}$ &$\abs{13}$ &$\abs{14}$ & $\abs{15}$ &$\abs{16}$ &$\abs{17}$\\ \hline \hline
0   &                1&       1&        &        &        &        &        &       &       &       &       &      &      &     &     &    &    &   \\
1   &                4&       4&       2&       1&        &        &        &       &       &       &       &      &      &     &     &    &    &   \\
2   &               20&      19&      13&       8&       3&       1&        &       &       &       &       &      &      &     &     &    &    &   \\
3   &              104&      98&      76&      52&      28&      13&       4&      1&       &       &       &      &      &     &     &    &    &   \\
4   &              556&     526&     434&     319&     201&     111&      50&     19&      5&      1&       &      &      &     &     &    &    &   \\
5   &             3032&    2887&    2470&    1910&    1316&     811&     436&    205&     80&     26&      6&     1&      &     &     &    &    &   \\
6   &            16778&   16073&   14085&   11304&    8259&    5489&    3284&   1763&    833&    344&    119&    34&     7&    1&     &    &    &   \\
7   &            93872&   90386&   80584&   66514&   50680&   35588&   22912&  13476&   7176&   3437&   1456&   539&   168&   43&    8&   1&    &   \\
8   &           529684&  512128&  462620&  390266&  306958&  224758&  152744&  96065&  55633&  29521&  14232&  6182&  2382&  802&  228&  53&   9&  1\\ \hline
OEIS&          A026520& A026521 & A026522 & A026523 & A026524 & & & & & & & & & & & & & \\	\hline
\end{tabular}
\end{adjustbox}
\caption{Number of paths terminating at point (n,m) for $\GDWSet$}
\label{table:recurrence5}
\end{table}

\begin{table}[H]
\centering
\begin{adjustbox}{max width=\textwidth}
\begin{tabular}{|r||r|r|r|r|r|r|r|r|r|r|r|r|r|r|r|r|r|r|}
  \hline
  $n\backslash m$ &0 &1 &2 &3 &4 &5 &6 &7 &8 &9 &10 &11 &12 &13 &14 & 15 &16 &17\\ \hline \hline
0    &               1&        &        &        &        &        &        &        &        &        &        &        &        &        &        &        &        &               \\
1    &               1&       1&       1&        &        &        &        &        &        &        &        &        &        &        &        &        &        &               \\
2    &               3&       4&       4&       2&       1&        &        &        &        &        &        &        &        &        &        &        &        &               \\
3    &              11&      17&      18&      13&       8&       3&       1&        &        &        &        &        &        &        &        &        &        &               \\
4    &              46&      76&      85&      72&      51&      28&      13&       4&       1&        &        &        &        &        &        &        &        &               \\
5    &             207&     355&     415&     384&     300&     196&     110&      50&      19&       5&       1&        &        &        &        &        &        &               \\
6    &             977&    1716&    2076&    2034&    1705&    1236&     785&     430&     204&      80&      26&       6&       1&        &        &        &        &               \\
7    &            4769&    8519&   10584&   10801&    9541&    7426&    5145&    3165&    1729&     826&     343&     119&      34&       7&       1&        &        &               \\
8    &           23872&   43192&   54798&   57672&   53038&   43504&   32151&   21456&   12937&    7008&    3394&    1448&     538&     168&      43&       8&       1&               \\ \hline
OEIS&  Bisect A214938& & & & & & & & & & & & & & & & & \\	\hline
\end{tabular}
\end{adjustbox}
\caption{Number of paths terminating at point (n,m) for $\DWRSet$}
\label{table:recurrence7}
\end{table}

\begin{table}[H]
\centering
\begin{adjustbox}{max width=\textwidth}
\begin{tabular}{|r||r|r|r|r|r|r|r|r|r|r|r|r|r|r|r|r|r|r|}
  \hline
  $n\backslash m$ &0 &1 &2 &3 &4 &5 &6 &7 &8 &9 &10 &11 &12 &13 &14 & 15 &16 &17\\ \hline \hline
0    &    1&       1&        &        &        &        &        &        &        &        &        &        &        &        &        &        &        &               \\
1    &    2&       3&       2&       1&        &        &        &        &        &        &        &        &        &        &        &        &        &               \\
2    &    7&      11&      10&       7&       3&       1&        &        &        &        &        &        &        &        &        &        &        &               \\
3    &   28&      46&      48&      39&      24&      12&       4&       1&        &        &        &        &        &        &        &        &        &               \\
4    &  122&     207&     233&     208&     151&      92&      45&      18&       5&       1&        &        &        &        &        &        &        &               \\
5    &  562&     977&    1154&    1099&     880&     606&     356&     179&      74&      25&       6&       1&        &        &        &        &        &               \\
6    & 2693&    4769&    5826&    5815&    4975&    3726&    2451&    1419&     714&     310&     112&      33&       7&       1&        &        &        &               \\
7    &13288&   23872&   29904&   30926&   27768&   22112&   15736&   10039&    5720&    2898&    1288&     496&     160&      42&       8&       1&        &               \\
8    &67064&  121862&  155662&  165508&  154214&  128693&   97111&   66544&   41401&   23339&   11850&    5380&    2154&     749&     219&      52&       9&               \\ \hline
OEIS& Bisect A214938& Bisect A214938& & & & & & & & & & & & & & & & \\	\hline
\end{tabular}
\end{adjustbox}
\caption{Number of paths terminating at point (n,m) for $\DWSet$}
\label{table:recurrence8}
\end{table}

\begin{table}[H]
\centering
\begin{adjustbox}{max width=\textwidth}
\begin{tabular}{|r||r|r|r|r|r|r|r|r|r|r|r|r|r|r|r|r|r|r|}
  \hline
  $n\backslash m$ &0 &$\abs{1}$ &$\abs{2}$ &$\abs{3}$ &$\abs{4}$ &$\abs{5}$ &$\abs{6}$ &$\abs{7}$ &$\abs{8}$ &$\abs{9}$ &$\abs{10}$ &$\abs{11}$ &$\abs{12}$ &$\abs{13}$ &$\abs{14}$ & $\abs{15}$ &$\abs{16}$ &$\abs{17}$\\ \hline \hline
0   &     1&        &        &        &        &        &        &        &        &        &        &        &        &        &        &        &        &                \\
1   &      &       2&        &       1&        &        &        &        &        &        &        &        &        &        &        &        &        &                \\
2   &    11&        &       9&        &       4&        &       1&        &        &        &        &        &        &        &        &        &        &                \\
3   &      &      58&        &      41&        &      20&        &       6&        &       1&        &        &        &        &        &        &        &                \\
4   &   343&        &     300&        &     200&        &      99&        &      35&        &       8&        &       1&        &        &        &        &                \\
5   &      &    1943&        &    1561&        &    1000&        &     503&        &     193&        &      54&        &      10&        &       1&        &                \\
6   & 11837&        &   10794&        &    8167&        &    5094&        &    2588&        &    1051&        &     331&        &      77&        &      12&                \\
7   &      &   69670&        &   59357&        &   42968&        &   26278&        &   13453&        &    5686&        &    1944&        &     520&        &     102        \\
8   &430819&        &  401490&        &  324653&        &  227151&        &  136849&        &   70470&        &   30692&        &   11136&        &    3277&                 \\ \hline
\end{tabular}
\end{adjustbox}
\caption{Number of paths terminating at point (n,m) for $\GCWRSet$}
\label{table:recurrence9}
\end{table}

\begin{table}[H]
\centering
\begin{adjustbox}{max width=\textwidth}
\begin{tabular}{|r||r|r|r|r|r|r|r|r|r|r|r|r|r|r|r|r|r|r|}
  \hline
  $n\backslash m$ &0 &$\abs{1}$ &$\abs{2}$ &$\abs{3}$ &$\abs{4}$ &$\abs{5}$ &$\abs{6}$ &$\abs{7}$ &$\abs{8}$ &$\abs{9}$ &$\abs{10}$ &$\abs{11}$ &$\abs{12}$ &$\abs{13}$ &$\abs{14}$ & $\abs{15}$ &$\abs{16}$ &$\abs{17}$\\ \hline \hline
0   &      1&        &        1&        &        &        &        &        &        &        &        &        &        &        &        &        &        &                \\
1   &       &       5&         &       3&        &       1&        &        &        &        &        &        &        &        &        &        &        &                \\
2   &     29&        &       24&        &      14&        &       5&        &       1&        &        &        &        &        &        &        &        &                \\
3   &       &     157&         &     119&        &      67&        &      27&        &       7&        &       1&        &        &        &        &        &                \\
4   &    943&        &      843&        &     599&        &     334&        &     142&        &      44&        &       9&        &       1&        &        &                \\
5   &       &    5447&         &    4504&        &    3064&        &    1696&        &     750&        &     257&        &      65&        &      11&        &       1        \\
6   &  33425&        &    30798&        &   24055&        &   15849&        &    8733&        &    3970&        &    1459&        &     420&        &      90&                \\
7   &       &  198697&         &  171995&        &  128603&        &   82699&        &   45417&        &   21083&        &    8151&        &    2556&        &     612        \\
8   &1233799&        &  1156962&        &  953294&        &  688653&        &  434470&        &  238012&        &  112290&        &   45078&        &   14997&                \\ \hline
\end{tabular}
\end{adjustbox}
\caption{Number of paths terminating at point (n,m) for $\GCWSet$}
\label{table:recurrence10}
\end{table}

\begin{table}[H]
\centering
\begin{adjustbox}{max width=\textwidth}
\begin{tabular}{|r||r|r|r|r|r|r|r|r|r|r|r|r|r|r|r|r|r|r|}
  \hline
  $n\backslash m$ &0 &1 &2 &3 &4 &5 &6 &7 &8 &9 &10 &11 &12 &13 &14 & 15 &16 &17\\ \hline \hline
0     &     1&        &        &        &        &        &        &        &        &        &        &        &        &        &        &        &        &                \\
1     &      &       1&        &       1&        &        &        &        &        &        &        &        &        &        &        &        &        &                \\
2     &     4&        &       5&        &       3&        &       1&        &        &        &        &        &        &        &        &        &        &                \\
3     &      &      19&        &      23&        &      14&        &       5&        &       1&        &        &        &        &        &        &        &                \\
4     &    70&        &     110&        &     107&        &      66&        &      27&        &       7&        &       1&        &        &        &        &                \\
5     &      &     439&        &     626&        &     525&        &     320&        &     141&        &      44&        &       9&        &       1&        &                \\
6     &  1684&        &    2942&        &    3344&        &    2657&        &    1591&        &     734&        &     256&        &      65&        &      11&                \\
7     &      &   11977&        &   18503&        &   18006&        &   13560&        &    8068&        &    3830&        &    1441&        &     418&        &      88        \\
8     & 47083&        &   86936&        &  107498&        &   97295&        &   70074&        &   41378&        &   20083&        &    7968&        &    2529&                \\ \hline
\end{tabular}
\end{adjustbox}
\caption{Number of paths terminating at point (n,m) for $\CWRSet$}
\label{table:recurrence11}
\end{table}

\begin{table}[H]
\centering
\begin{adjustbox}{max width=\textwidth}
\begin{tabular}{|r||r|r|r|r|r|r|r|r|r|r|r|r|r|r|r|r|r|r|}
  \hline
  $n\backslash m$ &0 &1 &2 &3 &4 &5 &6 &7 &8 &9 &10 &11 &12 &13 &14 & 15 &16 &17\\ \hline \hline
0       &     1&        &       1&        &        &        &        &        &        &        &        &        &        &        &        &        &        &                \\
1       &      &       3&        &       3&        &       1&        &        &        &        &        &        &        &        &        &        &        &                \\
2       &    11&        &      15&        &      12&        &       5&        &       1&        &        &        &        &        &        &        &        &                \\
3       &      &      59&        &      77&        &      54&        &      25&        &       7&        &       1&        &        &        &        &        &                \\
4       &   221&        &     364&        &     378&        &     260&        &     125&        &      42&        &       9&        &       1&        &        &                \\
5       &      &    1463&        &    2155&        &    1921&        &    1276&        &     636&        &     236&        &      63&        &      11&        &       1        \\
6       &  5666&        &   10120&        &   11893&        &    9948&        &    6408&        &    3259&        &    1297&        &     395&        &      88&                \\
7       &      &   41417&        &   65179&        &   65562&        &   51697&        &   32727&        &   16850&        &    7041&        &    2338&        &     583        \\
8       &163799&        &  306433&        &  386557&        &  360673&        &  270875&        &  168856&        &   87798&        &   37971&        &   13366&                 \\ \hline
\end{tabular}
\end{adjustbox}
\caption{Number of paths terminating at point (n,m) for $\CWSet$}
\label{table:recurrence12}
\end{table}
\end{landscape}
\end{document}